\documentclass[11pt]{amsart}
\usepackage{amsmath}    
\usepackage{amsthm}     
\usepackage{amssymb}    
\usepackage{comment}

\usepackage{mathtools}
\usepackage{amsthm}
\usepackage{mathrsfs}
\usepackage{amsfonts}
\usepackage{amstext}
\usepackage{amsmath}
\usepackage{amssymb}
\usepackage{color,xcolor}
\usepackage{bm}
\usepackage{bbm}
\usepackage{graphicx}
\usepackage{xcolor}

\newtheorem{thm}{Theorem}[section]
\newtheorem{cor}[thm]{Corollary}
\newtheorem{lem}[thm]{Lemma}
\newtheorem{prop}[thm]{Proposition}
\theoremstyle{definition}

\newtheorem{rem}[thm]{Remark}

\numberwithin{equation}{section}

\begin{document}
	
	\title{Ascent and descent  of operators concerning Jordan products}
	
	\author{Xiaoshuo Feng, Bing Xu, Jinchuan Hou, Xiaofei Qi}

     \address[Feng]{School of Mathematics and Statistics, Shanxi University, Taiyuan 030006,  China}
     
       \address[Xu]{School of Computer and Information Technology, Shanxi University, Taiyuan 030006,  China}

      \address[Hou]{College of Mathematics, Taiyuan University of
      	Technology, Taiyuan 030024,  China}
      	
      	 \address[Qi]{School of Mathematics and Statistics, Shanxi University, Taiyuan 030006, China; Key Laboratory of Complex Systems and Data Science of Ministry of Education, Shanxi University, Taiyuan 030006,  China}
	
    \email{xiaofeiqisxu@aliyun.com;  xiaoshuofeng0919@163.com; 202212211015@email.sxu.edu.cn; jinchuanhou@aliyun.com  }


	\begin{abstract}
		
		 Let $\mathcal B(\mathcal X)$ be the algebra of all bounded linear operators on a Banach space   $\mathcal{X}$ over the  real or complex field $\mathbb F$ with $\dim\mathcal X \ge 3$. In this paper,
	we give several properties of the ascent  (descent) of  operators or Jordan products of operators in $\mathcal B(\mathcal X)$ and use these properties to characterize certain operators such as algebraic  operators, rank-one, rank-two operators, rank-one nilpotent operators as well as some relations between operators in terms of ascent (descent) of Jordan products.  As an application, we give a characterization, up to a multiplier of   nonlinear functional $f:\mathcal B(\mathcal X)\to \mathbb F^*$, of ring isomorphisms of $\mathcal B(\mathcal X)$ in the case $\dim \mathcal X=\infty$ and  Jordan isomorphisms of $\mathcal B(\mathcal X)$  in the case $\dim \mathcal X<\infty$ via the maps with range containing all  operators of rank at most three that preserve the ascent of Jordan products of operators. A slightly different  characterization in term of descent is also obtained.
		
	\end{abstract}
	\maketitle

	\section{Introduction}
	
	Let $\mathcal X$ be any Banach space over the real or complex field $\mathbb F$ and $\mathcal B(\mathcal X)$  the algebra of all bounded linear operators on  $\mathcal X$.  
	For any operator $T \in \mathcal B(\mathcal X)$, $\operatorname{ker}(T)$ and $\operatorname{ran}(T)$ stand for the kernel and the range of $T$, respectively. Recall that the ascent and the descent of $T$ are defined  respectively by \cite{FR1918}
	$$\operatorname{asc}(T) = \inf\{ n \in \mathbb{N} \cup \{0\} \colon  \operatorname{ker}(T^n) = \operatorname{ker}(T^{n+1}) \} $$and
	$$\operatorname{desc}(T) = \inf\{ n \in \mathbb{N} \cup \{0\} \colon  \operatorname{ran}(T^n) = \operatorname{ran}(T^{n+1}) \}. $$
	Here, the infimum over the empty set is taken to be infinite (\cite{DCL1980,VM2007}).
The ascent and the descent   play a decisive role in the spectral theory and are the basic bricks in the construction of   Fredholm operator theory (\cite{PA2004, SG2002, DCL1970,  DCL1980}).
The classical Riesz--Schauder theory asserts \cite{HH1982} that $\operatorname{asc}(\lambda I-K)=\operatorname{desc}(\lambda I-K)<\infty$ for every compact operator $K\in\mathcal B(\mathcal X)$ and every nonzero complex number $\lambda$. Also	
	  recall that an operator $A\in\mathcal B(\mathcal X)$ is of rank one if $\dim {\rm ran}(A)=1$. Every rank-one operator can be written as $A=x\otimes f$, defined by $(x\otimes f)y = f(y)x$ for any  $y\in\mathcal X$, for some vector $x\in \mathcal X$ and bounded linear functional $f\in\mathcal X^*$. Clearly,
$A=x\otimes f$ is nilpotent if and only if $A^2=0$, and in turn, if and only if $f(x)=\langle x,f\rangle=0$. Denote respectively by $\mathcal F_1(\mathcal X)$ and $\mathcal{N}_1(\mathcal X)$  the set of all rank-one operators on $\mathcal X$ and the set of all nilpotent operators in $\mathcal F_1(\mathcal X)$.

In recent years, there has been  considerable interest in discussing properties of operators involving the ascent and  the descent.
Many properties of bounded linear operators can be  characterized with the help of the ascent and the descent. The following well-known properties concerning ascent and descent are useful.
	
	\begin{prop}[\cite{RH2023, VM2007, AET1966}] \label{p1}
		Assume that $A \in \mathcal B(\mathcal X)$. Then the following assertions hold.

		{\rm (1)} $\operatorname{asc}(kA) =\operatorname{asc}(A)$ and $\operatorname{desc}(kA) =\operatorname{desc}(A)$ hold for all $k \in \mathbb F^*=\mathbb F\setminus \{0\}$.
		
		{\rm (2)}  $\operatorname{asc}(A) =\operatorname{asc}(T^{-1}AT)$ and $\operatorname{desc}(A) =\operatorname{desc}(T^{-1}AT)$ hold for all invertible operators $T\in\mathcal B(\mathcal X)$.

		{\rm (3)} $A$  is injective   $\Leftrightarrow \operatorname{asc}(A) = 0 \Leftrightarrow \operatorname{asc}(A^2) = 0$; $A$ is surjective $\Leftrightarrow \operatorname {desc}(A)=0\Leftrightarrow  \operatorname {desc}(A^2)=0$.

		{\rm (4)} If $A = A_1 \oplus A_2$, then
		$\operatorname{asc}(A) = \max\{\operatorname{asc}(A_1), \operatorname{asc}(A_2)\}$ and $\operatorname{desc}(A) = \max\{\operatorname{desc}(A_1),\\ \operatorname{desc}(A_2)\}$.
		
		{\rm (5)} If $\operatorname{asc}(A) < \infty$ and $\operatorname{desc}(A) < \infty$, then $\operatorname{asc}(A) = \operatorname{desc}(A)$ and $\mathcal X$ has a space decomposition  $\mathcal X=\operatorname {ker}(A^{\operatorname{asc}(A)})\oplus\operatorname{ran}(A^{\operatorname{asc}(A)})$.
		
		{\rm (6)} If $A$ is a $k$-nilpotent operator, then $\operatorname{asc}(A) = \operatorname{desc}(A)= k$.
		
		{\rm (7)} If  $A = x \otimes f $ is of rank one, then $\operatorname {asc}(A)=\operatorname{desc}(A)  \in\{1,2\}$;  and moreover,  $A \in \mathcal N_1(\mathcal{X}) \Leftrightarrow \operatorname{asc}(A) =2\Leftrightarrow\operatorname{desc}(A)= 2$.
	\end{prop}

One of our purposes in the present paper is to further investigate the properties associated with the ascent and the descent of bounded linear operators acting on Banach spaces. Specifically, we first explore the ascent (descent) of upper triangular block operator matrices and certain special algebraic operators  (Section 2). Subsequently, we establish characterizations for the ascent (descent) of rank-one and rank-two operators (Section 3). Building on these results, we proceed to characterize  features for some special operators  and operator relations by the ascent (descent) of Jordan products (Section 4).

 It is also a topic of considerable interest to study  the  related preservers about the ascent (descent) on operator algebras recently.
Assume that  $ \mathcal A$, $\mathcal B$ are any two operator algebras and  $\diamond$ is some product of elements $A,B$ in $\mathcal A$. We say that a map $ \phi : \mathcal A \to \mathcal B $   preserves the ascent (descent) of $A\diamond B$  if $\phi$ satisfies $\operatorname{asc}(A\diamond B)=\operatorname{asc}(\phi(A)\diamond\phi(B))$  ($\operatorname{desc}(A\diamond B)=\operatorname{desc}(\phi(A)\diamond\phi(B))$) for all  $A,B \in \mathcal A$. Bendaoud and Sarih in \cite{MB2009} showed that a surjective additive map $\phi $ from $ \mathcal B(\mathcal X)$ onto $ \mathcal B(\mathcal Y)$  (here, $\mathcal X$ and $\mathcal Y$ are infinite-dimensional complex Banach spaces) satisfies $\operatorname{asc}(A) = \operatorname{asc}(\phi(A))$ (resp. $\operatorname{desc}(A) = \operatorname{desc}(\phi(A))$) for all $A\in\mathcal B(\mathcal X)$ if and only if $\phi$ is  a scalar multiple  of a ring isomorphism (resp. a ring isomorphism or an anti-ring isomorphism).
Hosseinzadeh and Petek discussed  all  surjective maps $\phi$ on $ \mathcal B(\mathcal X)$ with $ \dim\mathcal X\geq 3 $ preserving the ascent (descent) of operator product $AB$ and preserving the ascent (descent) of Jordan semi-triple product $ABA$ in \cite{RH2023, RH2024},  respectively, and obtained the concrete form of such $\phi$.   Recently, Marzouki and Souilah in \cite{RM2026}  refined   the result in \cite{RH2023} for  the infinite-dimensional complex Banach space case  by weakening the surjectivity of  $\phi$.  For other related results, the reader can see  \cite{MM2014, MMV2014, MM2017,  MO2019, MO2020,  WJS2020} and the references therein. These results reveal that the ascent as well as descent have some rigid features that determine the  ring or Jordan ring structure of the operator algebra $\mathcal B(\mathcal X)$.

As an application of the results in Sections 2-4, another purpose of the present paper is to characterize the general preservers $\phi$ on $\mathcal B(\mathcal X)$  preserving the ascent  of operator Jordan products, that is, those $\phi$ satisfying asc$(\phi(A)\phi(B)+\phi(B)\phi(A))=$asc$(AB+BA)$ for all $A,B\in\mathcal B(\mathcal X)$. This will be discussed in Section 5.  We show that, under the assumption that the range of $\phi$ contains all operators with rank $\leq 3$, such preservers give a characterization of ring isomorphisms of $\mathcal B(\mathcal X)$ in the case $\dim \mathcal X=\infty$ or Jordan ring isomorphisms of $\mathcal B(\mathcal X)$ in the case $\dim \mathcal X<\infty$ up to a multiplier of nonlinear functional $f:\mathcal B(\mathcal X)\to \mathbb F^*$ (Theorem \ref{1}). A slightly different result for preservers of the  descent is also established (Theorem \ref{22}).

At the end of this section, we fix some more notions and notations.  In this paper, if there is no special illustration,  $\mathcal{X}$ is any Banach space over the  real or complex field $\mathbb F$ with $\dim\mathcal X\geq 3$ and $\mathcal X^*$ is its  dual space.  $\mathbb F^*=\mathbb F\setminus \{0\}$.
	Denote respectively by $\mathcal P_1(\mathcal{X})$, $\mathcal{N}_1(\mathcal{X})$ and $\mathcal F_1(\mathcal X)$ the set of all rank-one idempotent operators, the set of all rank-one nilpotent operators and the set of all rank-one operators in $\mathcal B(\mathcal X)$.   	For $n>1$, denote by $\mathcal F_n(\mathcal X)$ the set of all finite rank operators in $\mathcal B(\mathcal X)$  with rank not greater than $n$.  For a subspace $\mathcal M \subseteq \mathcal X$,   $\mathcal M^\perp  = \{ f \in \mathcal X^*: f(x) = 0  \text{ for all }  x \in \mathcal M \}$. If   $\operatorname{dim}\mathcal X=n < \infty$, then $\mathcal X$ is identified with $\mathbb F^n$,  and $\mathcal{B}(\mathcal{X})$ is identified with $\mathcal M_n(\mathbb F)$. Let $\{e_i\}_{i=1}^{n}$ be an orthonormal basis of $\mathbb F^n$ and $A \in \mathcal M_n(\mathbb F)$. With respect to this basis, $A$ has a matrix representation $A = [a_{ij}]$ with $a_{ij} = \langle Ae_j, e_i \rangle$.
$A^{\rm tr}=[a_{ji}]$ is the transpose of  $A$. Let $E_{ij} \in \mathcal M_n(\mathbb F)$ be the matrix with $(i, j)$-th entry  1 and all other entries  0.  $J_k(\lambda) \in \mathcal M_k(\mathbb F)$ denotes  the Jordan block  matrix given by  $J_k(\lambda) = \lambda I_k + N_k$, where $\lambda \in \mathbb{F} $, $I_k$  is the  $k \times k$  identity matrix and  $N_k = \sum_{r=1}^{k-1} E_{r,r+1}$.
 $\operatorname{diag}(a_1, a_2, \dots, a_n)$ denotes a diagonal matrix with   main diagonal  entries $a_1, a_2, \dots, a_n$ and $\operatorname{diag}\{A_1, A_2, \dots, A_n\}$ denotes a block diagonal matrix with  square matrices $A_1, A_2, \dots, A_n$ on its main diagonal entries.

\section{ Ascent  and descent of algebraic  operators and  triangular operator  matrices}

In this section, we  discuss the ascent  (descent) of   upper triangular operator matrices and  algebraic operators.

The following results are devoted to studying the relationship between the ascent of an upper triangular block operator matrix and the ascent of its diagonal blocks.

\begin{prop} \label{upper triangular matrix}  Let $\mathcal X, \mathcal Y$ be two Banach spaces, $S =\begin{pmatrix}
		A&B\\
		0&C
	\end{pmatrix}\in{\mathcal B}(\mathcal X\oplus\mathcal Y)$ with $A \in \mathcal B(\mathcal X)$, $B\in \mathcal B(\mathcal Y, \mathcal X)$ and $C \in \mathcal B(\mathcal Y)$. If $\operatorname{asc}(A) = n$, then $\operatorname{asc}(S) \ge n$; and moreover, if $C=0$, then $n \leq \operatorname{asc}(S) \leq n+1 $.
\end{prop}

\begin{proof}
	
	Assume $\operatorname{asc}(A) = n$.
	If $n=0$, the assertion is immediate. Hence assume $n\ge1$. As $\operatorname{asc}(A) = n$, one has $\operatorname{ker}(A^{n - 1}) \subsetneq \operatorname{ker}(A^n) = \operatorname{ker}(A^{n + 1})$. So, there exists  some $x \in \mathcal X$ such that $A^{n - 1}x \neq 0$ but $A^nx= 0$.  By a direct computation, we have

$$f(x^2)$$

	$${S^{n - 1}}\begin{pmatrix}
		x\\
		0
	\end{pmatrix}= \begin{pmatrix}
		A^{n-1}&\sum_{i=0}^{n-2}A^{n-2-i}BC^i\\
		0&C^{n-1}
	\end{pmatrix}\begin{pmatrix}
		x\\
		0
	\end{pmatrix}=\begin{pmatrix}
		A^{n - 1}x\\
		0
	\end{pmatrix} \ne 0,$$
	but
	$${S^{n}}\begin{pmatrix}
		x\\
		0
	\end{pmatrix}= \begin{pmatrix}
		A^n&\sum_{i=0}^{n-1}A^{n-1-i}BC^i\\
		0&C^n
	\end{pmatrix}\begin{pmatrix}
		x\\
		0
	\end{pmatrix}=\begin{pmatrix}
		A^nx\\
		0
	\end{pmatrix} = 0.$$   Thus  $\operatorname{ker}(S^{n - 1})\subsetneq \operatorname{ker}(S^n)$ and $\operatorname{asc}(S) \ge n$.
	
	Furthermore, if $C=0$, then for any $\begin{pmatrix}
		x\\
		y
	\end{pmatrix}\in \operatorname{ker}(S^{n+2})$ with  $x \in \mathcal X$ and $y \in \mathcal Y$, we have
	$$0=S^{n + 2}\begin{pmatrix}
		x\\
		y
	\end{pmatrix}= \begin{pmatrix}
		A^{n+2}&A^{n+1}B\\
		0&0
	\end{pmatrix}\begin{pmatrix}
		x\\
		y
	\end{pmatrix}=\begin{pmatrix}
		A^{n+1}(Ax+By)\\
		0
	\end{pmatrix},$$
	which implies  $Ax+By \in \operatorname{ker}(A^{n+1})$. On the other hand, as $\operatorname{asc}(A)= n $, one  gets $A^{n}(Ax+By)= A^{n+1}x+A^{n}By=0$, and hence
	$$S^{n + 1}\begin{pmatrix}
		x\\
		y
	\end{pmatrix}= \begin{pmatrix}
		{{A^{n+1}}}&{{A^{n}B}}\\
		0&0
	\end{pmatrix}\begin{pmatrix}
		x\\
		y
	\end{pmatrix}=\begin{pmatrix}
		A^{n+1}x+A^{n}By\\
		0
	\end{pmatrix} = 0.$$  It follows that  $\operatorname{asc}(S) \leq n+1 $.
\end{proof}

Similarly, for the descent of upper triangular operator matrix, we have

\begin{prop}   Let $\mathcal X, \mathcal Y$ be any Banach spaces, $S =\begin{pmatrix}
		A&B\\
		0&C
	\end{pmatrix}\in{\mathcal B}(\mathcal X\oplus\mathcal Y)$ with $A \in \mathcal B(\mathcal X)$, $B\in \mathcal B(\mathcal Y, \mathcal X)$ and $C \in \mathcal B(\mathcal Y)$.  If $\operatorname{desc}(C) = n$, then $\operatorname{desc}(S) \ge n$;  moreover, if $A=0$, then $n \leq \operatorname{desc}(S) \leq n+1 $.
\end{prop}

 The following proposition gives a  useful formula relating the ascent of an operator to the ascent and descent of its powers.

	\begin{prop}\label{ascent power}	For any $A\in\mathcal B(\mathcal X)$, we have
		$\operatorname{asc}(A^n)
		=\left\lceil\frac{\operatorname{asc}(A)}{n}\right\rceil$ and
	$\operatorname{desc}(A^n)
		=\left\lceil\frac{\operatorname{desc}(A)}{n}\right\rceil$,
		where  $n\in\mathbb N$, $\lceil r\rceil$ is the smallest integer greater than or equal to $ r$ with $\lceil\infty\rceil=\infty$.
		\end{prop}
	
	\begin{proof}

	Assume $n\in\mathbb N$.
Note that, for every positive integer $k$, we have
		$$\begin{aligned}
			\operatorname{ker}((A^n)^k)=\operatorname{ker}((A^n)^{k+1})
			\iff \operatorname{ker} (A^{nk})=\operatorname{ker} (A^{nk+1})\iff nk\geq\operatorname{asc}(A)
		\end{aligned}$$
			and
		$$	\begin{aligned}
			\operatorname{ran}((A^n)^k)=\operatorname{ran}((A^n)^{k+1})
			\iff \operatorname{ran}(A^{nk})
			=\operatorname{ran}(A^{nk+1})
			\iff nk\geq\operatorname{desc}(A).
		\end{aligned}$$
		Consequently, we conclude
		$$\operatorname{asc}(A^n)
			=\inf\{k\in\mathbb N\cup\{0\}:
			nk\geq\operatorname{asc}(A)\}
			=\left\lceil\frac{\operatorname{asc}(A)}{n}\right\rceil$$
		and
	$$\operatorname{desc}(A^n)
			=\inf\{k\in\mathbb N\cup\{0\}:
			nk\geq\operatorname{desc}(A)\}\\
			=\left\lceil\frac{\operatorname{desc}(A)}{n}\right\rceil,$$
	where the infimum of the empty set is taken to be
		$\infty$. This completes the proof.
	\end{proof}

	Recall that an operator $T \in \mathcal B(\mathcal X)$ is said to be algebraic if there exists a non-zero  polynomial $p(t)$ such that $p(T) = 0$; is said to be locally algebraic if  for each $x \in  \mathcal{X} $, there exists a non-zero polynomial $p_x(t)$ such that $p_x(T)x = 0 $; is said to be algebraic  of degree $n$ if  there exists a nonzero polynomial  $p(t)$  with degree  $n$  such that  $p(T) = 0$, and there is no nonzero polynomial of degree less than  $n$  annihilating  $T$.   Obviously,   any  algebraic operator is necessarily a locally algebraic operator, and the converse is true by Kaplansky theorem. All finite rank operators are algebraic. For more details, the readers may refer to \cite{HR1973}.	
	
  The following result gives a property of non-scalar  algebraic operators in terms of ascent (descent) and Jordan products.

\begin{prop}\label{scalar perturbation}
	Assume that $A \in \mathcal B(\mathcal X)\setminus{\mathbb F}I$ is an algebraic operator. Then, for any  $\omega\in \mathbb F^*$,  there  exists  a finite rank algebraic
	operator with degree  not 2, or   a rank-one operator  $T \in \mathcal B(\mathcal X)$ such that
	$\operatorname{asc}(AT + TA) \neq \operatorname{asc}(AT + TA+\omega T)$ and $\operatorname{desc}(AT + TA) \neq \operatorname{desc}(AT + TA+\omega T)$.
\end{prop}

To prove Proposition \ref{scalar perturbation},  the following proposition and three lemmas are needed.

\begin{prop}[\cite{RA2013}]\label{D}
	Let $\mathcal Y$ and $\mathcal Z$ be two nontrivial closed subspaces such that $\mathcal X = \mathcal Y \oplus \mathcal Z$, and let $T \in \mathcal{B}(\mathcal X)$ have the operator matrix of the form
	\[
	T = \begin{pmatrix}
		A & B \\
		0 & C
	\end{pmatrix} \quad \mbox{\rm or}\quad T = \begin{pmatrix}
		A & 0 \\
		D & C
	\end{pmatrix},
	\]
	where $C$ is invertible. Then  $\operatorname{asc}(T)=\operatorname{desc}(T)=p < \infty $ if and only if  $\operatorname{asc}(A)=\operatorname{desc}(A)=p < \infty $.
\end{prop}

\begin{lem}\label{L1}
	Assume that
	$T_a = \begin{pmatrix}
		-a&a^2\\
		2&-a
	\end{pmatrix}$ and  $C_b = \begin{pmatrix}
		1 & b  \\
		0 & 1
	\end{pmatrix}$  in $\mathcal M_2(\mathbb F)$, where $a \in \mathbb F^*$ and $b \in \mathbb F$.
	Then $D_a=\operatorname{diag}\{T_a, -2a\}$  is algebraic with degree not  2 and
	
	{\rm (1)} $\operatorname{asc}( C_a T_a+T_a C_a) =\operatorname{desc}( C_a T_a+T_a C_a) = 2$;
	
	{\rm (2)}   $\operatorname{asc}(C_b T_a+T_a C_b) =\operatorname{desc}( C_b T_a+T_a C_b) =\begin{cases}1 &{\rm if}\ b=-a,\\
		0&{\rm if}\ \ b\not=\pm a.\end{cases}$
\end{lem}

\begin{proof}
	Firstly, we prove that the degree of $D_a=\operatorname{diag}\{T_a, -2a\}$  is not   2, that is, there are no scalars $\lambda_{1}, \lambda_{2}\in\mathbb F$ such that $D_a^2 + \lambda_{1}D_a + \lambda_2 I = 0$.  Otherwise, as
	$$\begin{pmatrix}
		T_a&0\\
		0 &-2a\\
	\end{pmatrix}^2= \begin{pmatrix}
		3a^2  & -2a^3 & 0 \\
		-4a & 3a^2& 0 \\
		0 & 0 & 4a^2
	\end{pmatrix}$$
	one gets $3a^2 -a\lambda_{1}+\lambda_{2}=0$, $-4a+2\lambda_{1}=0$ and $4a^2 -2a\lambda_{1}+\lambda_{2}=0$, which implies $a=0$, a contradiction.
	
	Note that
	$$C_a T_a+T_a C_a =\begin{pmatrix}
		0&0\\
		4&0\\
	\end{pmatrix} \text{ and }
	C_b T_a+T_a C_b= \begin{pmatrix}
		2b-2a&2a^2-2ab\\
		4&2b-2a
	\end{pmatrix}.$$
	By   Proposition \ref{p1}(6),  one has  $\operatorname{asc}( C_a T_a+T_a C_a) = \operatorname{desc}( C_a T_a+T_a C_a) =2$. Now, assume that $b\not=a$. If $b \neq -a$, then $C_b T_a+T_a C_b$ is invertible, which and    Proposition \ref{p1}(3) imply  $\operatorname{asc}(C_b T_a+T_a C_b) =\operatorname{desc}( C_b T_a+T_a C_b) =0$. If $b = -a$, it is easily seen that $\operatorname{rank}( C_b T_a+T_a C_b)=1$ and $( C_b T_a+T_a C_b)^2 \not=0$. So $\operatorname{asc}( C_b T_a+T_a C_b) = 1$, and hence $\operatorname{desc}( C_b T_a+T_a C_b) =1$.
	
	The proof is finished.
\end{proof}

\begin{lem}\label{L2}
	Assume that
	$T_a = \begin{pmatrix}
		a(a+1)&2a\\
		-2a&-(a+1)
	\end{pmatrix}$ and $C_b = \begin{pmatrix}
		1 & 0 \\
		0 & b
	\end{pmatrix} $  in $\mathcal M_2(\mathbb F)$, where $a \in \mathbb F \setminus \{0, \pm1 \}$ and $b \in \mathbb F$.
	Then $D_a=\operatorname{diag}\{T_a, a-1\}$ is algebraic with degree not  2 and
	
	{\rm (1)} $\operatorname{asc}( C_a T_a+T_a C_a) =\operatorname{desc}( C_a T_a+T_a C_a) = 2$;
	
	{\rm (2)}   $\operatorname{asc}(C_b T_a+T_a C_b)=\operatorname{desc}( C_b T_a+T_a C_b) =\begin{cases}1 &{\rm if}\ b=\frac{1}{a},\\
		0&{\rm if}\ \ b\not\in\{\frac{1}{a},a\}.\end{cases}$
\end{lem}

\begin{proof}
	By a similar computation to that in the proof of Lemma \ref{L1},  one can show that $D_a=\operatorname{diag}\{T_a, a-1\}$  is algebraic with  degree not  2.
	
	Note that
	$$C_a T_a+T_a C_a =2a(a+1)\begin{pmatrix}
		1&1\\
		-1&-1\\
	\end{pmatrix} \text{ and }
	C_b T_a+T_a C_b= \begin{pmatrix}
		2a(a+1)&2a(b+1)\\
		-2a(b+1)&-2b(a+1)
	\end{pmatrix}.$$
	By   Proposition \ref{p1}(6),  one gets  $\operatorname{asc}( C_a T_a+T_a C_a) = \operatorname{desc}( C_a T_a+T_a C_a) =2$. Now, assume that $b\not=a$. In addition,  it is clear that  $\operatorname{det}(C_b T_a+T_a C_b) \neq 0$ if and only if  $b \neq \frac{1}{a}$. By Proposition \ref{p1}(3), it follows that
	$\operatorname{asc}( C_b T_a+T_a C_b)=\operatorname{desc}( C_b T_a+T_a C_b)=0$ if $b \neq \frac {1}{a} $.
	If $b = \frac {1}{a} $, one gets $\operatorname{rank}( C_b T_a+T_a C_b)=1$ and $( C_b T_a+T_a C_b)^2 \not=0$,  and so  $\operatorname{asc}( C_b T_a+T_a C_b) =\operatorname{desc}( C_b T_a+T_a C_b)= 1$.
	
	The proof is finished.
\end{proof}

\begin{lem}\label{L3}
	Assume that
	$T_a =\begin{pmatrix}
		0&{\frac{1}{2}}&{\frac{1}{{a + 1}}}\\
		{\frac{1}{{2a - 2}}}&{\frac{1}{{2a - 2}}}&{\frac{1}{{{a^2} - 1}}}\\
		{\frac{{ - 1}}{{{a^2} - 1}}}&{\frac{{ - 1}}{{{a^2} - 1}}}&{\frac{{ - 1}}{{2{a^2} - 2a}}}
	\end{pmatrix}$ and $C_b = \begin{pmatrix}
		1 & 0 & 0 \\
		0 & 1 & 0 \\
		0 & 0 & b
	\end{pmatrix}$ in $\mathcal M_3(\mathbb F)$, where $a \in \mathbb F \setminus \{0, \pm1 \}$ and  $b \in \mathbb F$.
	Then  $T_a$ is algebraic with degree not 2,  and
	
	{\rm (1)} $\operatorname{asc}( C_a T_a+T_a C_a) =\operatorname{desc}( C_a T_a+T_a C_a) = 3$;
	
	{\rm (2)} $\operatorname{asc}(C_b T_a+T_a C_b)=\operatorname{desc}( C_b T_a+T_a C_b) =\begin{cases}1 &{\rm if}\ b=\frac{1}{a},\\
		0&{\rm if}\ \ b\not\in\{\frac{1}{a},a\}.\end{cases}$
\end{lem}

\begin{proof}
	By a similar computation to that in the proof of Lemma \ref{L1},  one can verify that $T_a$ is algebraic with  degree not  2.
	
	Next, for any $b\in\mathbb F$ with $b\not=a$, we have $$C_a T_a + T_a C_a = \begin{pmatrix}
		0 & 1 & 1 \\
		\frac{1}{a-1} & \frac{1}{a-1} & \frac{1}{a-1} \\
		-\frac{1}{a-1} & -\frac{1}{a-1} & -\frac{1}{a-1}
	\end{pmatrix} \text{ and } C_bT_a+T_a C_b= \begin{pmatrix}
		0 & 1 & \frac{b+1}{a+1} \\
		\frac{1}{a-1} & \frac{1}{a-1} & \frac{b+1}{a^2-1} \\
		-\frac{b+1}{a^2-1} & -\frac{b+1}{a^2-1} & -\frac{b}{a^2-a}
	\end{pmatrix}.$$
	A direct computation gives
	$$(C_a T_a+T_a C_a)^2 \neq 0,\ \ (C_a T_a+T_a C_a)^3=0,$$ and $$ \operatorname{det}(C_b T_a+T_a C_b)= \frac{{b{{( {a + 1} )}^2}-a{{({b + 1})}^2}}}{{a{{({a - 1})}^2}{{( {a + 1})}^2}}}.$$
	By  Proposition \ref{p1}(6),  we see that $\operatorname{asc}( C_a T_a+T_a C_a)=3$.  In addition,  it is obvious that  $\operatorname{det}(C_b T_a+T_a C_b) \neq 0$ if and only if  $b \neq \frac{1}{a}$. This fact and Proposition \ref{p1}(3) imply that
	$\operatorname{asc}( C_b T_a+T_a C_b) = 0$ if $b \neq \frac {1}{a} $.
	If $b = \frac {1}{a} $,  a simple computation gives  $\operatorname{ker}(C_b T_a + T_a C_b)=\operatorname{ker}((C_b T_a + T_a C_b)^2)$, and so  $\operatorname{asc}( C_b T_a+T_a C_b) = 1$.
	
	As the descent and the ascent of any matrix are the same,  the proof is completed.
\end{proof}

\begin{proof}[Proof of Proposition \ref{scalar perturbation}]
	 For arbitrarily given $\omega\in\mathbb F^*$, we have to show that there exists a operator $T$  such that $T$ is finite rank algebraic
	 with degree  not 2 or  of rank-one and satisfies $\operatorname{asc}(AT + TA) \neq \operatorname{asc}(AT + TA+\omega T)$. We check this   by considering the following two cases.	
	
	{\bf Case 1.} $\mathbb F=\mathbb C$.
	
	Let $p(\lambda) = \prod_{i=1}^{n} (\lambda - \lambda_i)^{\alpha_i}$ be the minimal polynomial of $A$  such that $p(A) = 0$. Here,  $\{\lambda_1, \lambda_2, \dots, \lambda_n\}$ is the spectrum of $A$.  By \cite{VM2007} and \cite[Theorem V.11.3]{DCL1980},
	$\mathcal X$ has a space decomposition
	$\mathcal X = \bigoplus_{i=1}^{n} \operatorname{ker}((A - \lambda_iI)^{\alpha_i})$ and 	 $A$ can be written as $$A=\bigoplus_{i=1}^{n}
	(\lambda_iI|_{\operatorname{ker}((A - \lambda_iI)^{\alpha_i})}+N_i), \quad N_i^{\alpha_i}=0.$$
	Since each $N_i$ is nilpotent,  the restriction $A|_{\operatorname{ker}((A - \lambda_iI)^{\alpha_i})}$ admits a decomposition $A|_{\operatorname{ker}((A - \lambda_iI)^{\alpha_i})}=\bigoplus_j J_{r_{ij}}(\lambda_i)$, where $J_{r_{ij}}(\lambda_i)$ is an Jordan matrix of degree $r_{ij}\leq \alpha_i<\infty$ for each $j$, which  acts on a finite-dimensional
	$A$-invariant subspace $\mathcal M_{ij}\subseteq \operatorname{ker}((A - \lambda_iI)^{\alpha_i})$. 
	We now consider the following three subcases according
	to $\max\{r_{ij}\}$.

	{\bf Subcase 1.1. }   $\max\{r_{ij}\} \geq 3$.
	
	Without loss of generality, assume that
	$r_{11}=\max\{r_{ij}\}\geq 3$. Then $A$  can be written as
	$$A=J_{r_{11}}(\lambda_1) \oplus A_1,$$
	where $A_1$ denotes the restriction of $A$ to a closed $A$-invariant complementary subspace.
	
	If  $\lambda_1=0$, by taking $T=E_{11}\oplus 0$, one gets $AT+TA=E_{12}\oplus 0$ and $AT+TA+\omega T=(\omega E_{11}+ E_{12})\oplus 0$. So $\operatorname{asc}(AT + TA) \not=\operatorname{asc}(AT + TA+\omega T)$.  If  $\lambda_1+\frac{\omega}{2}=0$, putting $T=E_{11}\oplus 0$ leads to $AT+TA=(2\lambda_1 E_{11}+ E_{12})\oplus 0$ and $AT+TA+\omega T=E_{12}\oplus 0$, which implies $\operatorname{asc}(AT + TA) \not=\operatorname{asc}(AT + TA+\omega T)$.
	If $\lambda_1 \not\in\{0, -\frac{\omega}{2}\}$,  we write $$A= \begin{pmatrix}
		J_2(\lambda_1)& A_{12}  \\
		0 & J_{r_{11}-2}(\lambda_1)
	\end{pmatrix}\oplus A_1.$$  By taking $T=\begin{pmatrix}
		T_c& 0  \\
		0 & -2cI_{r_{11}-2}
	\end{pmatrix}\oplus 0$  with $D_c$ ($c=\frac{1}{\lambda_1}\not=0$) stated as in Lemma \ref{L1}.
	one gets
	$$AT+TA=\begin{pmatrix}
		J_2(\lambda_1)T_c+T_cJ_2(\lambda_1)	&T_cA_{12}-2cA_{12} \\
		0 & -4cJ_{r_{11}-2}(\lambda_1)
	\end{pmatrix}  \oplus 0$$
	and
	$$AT+TA+\omega T=\begin{pmatrix}
		J_2(\lambda_1)T_c+T_cJ_2(\lambda_1)+\omega T_c &T_cA_{12}-2cA_{12}  \\
		0 & -4c(J_{r_{11}-2}(\lambda_1) +  \frac{\omega}{2}I_{r_{11}-2})
	\end{pmatrix}  \oplus 0.$$
	Since $J_{r_{11}-2}(\lambda_1) $ and $J_{r_{11}-2}(\lambda_1) + \frac{\omega}{2}I_{r_{11}-2}$ are invertible,  by   Proposition \ref{p1}(3)-(4) and Proposition \ref{D}, it is easy to check $\operatorname{asc}(AT + TA) \not=\operatorname{asc}(AT + TA+\omega T)$.
	
	{\bf Subcase 1.2.} $\max\{r_{ij}\}= 2$.

	If   $r_{ij}=2$ for  exactly one pair  $(i,j)$ and $r_{kl}=1$ for all other pairs  $(k,l)$, then  we may write
	$$A= \operatorname{diag}\{J_2(\xi_1),\xi_2\} \oplus A_1 \ \ {\rm with} \ \ \xi_1, \xi_2\in \sigma(A).$$
	If $\xi_1 \in\{0, -\frac{\omega}{2}\}$, by taking $T =E_{11}\oplus 0$, it is easy to check
	$\operatorname{asc}(AT + TA) \not=\operatorname{asc}(AT + TA+\omega T)$.
	If $\xi_1 \not\in\{0, -\frac{\omega}{2}\}$, then we have $A= \xi_1 \operatorname{diag}\{J_2(1)+(\frac{1}{\xi_1}-1)E_{12}, \frac{\xi_2}{\xi_1} \} \oplus A_1$. Taking $T=\operatorname{diag}\{T_c, -2c\} \oplus 0$ with $c=\frac{1}{\xi_1} \not=0$ as stated in Lemma \ref{L1} and applying Lemma \ref{L1}, one achieves  $\operatorname{asc}(AT + TA) \neq \operatorname{asc}(AT + TA+\omega T)$.

	If there exist at least two pairs $(i,j)$ and $(i',j')$ such that  $r_{ij}=r_{i'j'}=2$, then
	$A$ at least contains two  $2 \times 2 $ Jordan blocks. Still,
	we may write
	$$A=\operatorname{diag}\{J_2(\mu_1),J_2(\mu_2)\} \oplus A_2  \ \ {\rm with} \ \ \mu_1, \mu_2\in \sigma(A).$$
	Similarly,	one can take suitable $T$  such that $\operatorname{asc}(AT + TA) \neq \operatorname{asc}(AT + TA+\omega T)$.
	
	{\bf Subcase 1.3.} $\max\{r_{ij}\}= 1$.
	
	In this case,  $A$ is a diagonal operator and  assume that
	$$A=\operatorname{diag}(\eta_1,\eta_2, \eta_3) \oplus A_1\ \ {\rm with}\ \ \eta_1, \eta_2, \eta_3 \in \sigma(A).$$

	If exactly two of $\eta_1,\eta_2,\eta_3$ coincide, we assume  $\eta_1=\eta_2$.	If $\eta_1+ \eta _3=0$,  putting a rank-one operator $T=E_{13}\oplus 0$ gives $AT + TA= 0$, which yields  $\operatorname{asc}(AT + TA)=1$ and $\operatorname{asc}(AT + TA+\omega T)= \operatorname{asc}(\omega E_{13} \oplus 0) =2$.
	If  $\eta_1+\eta_3\not=0$ and $\eta _1=-\frac{\omega}{2}$,  putting $T=E_{12}\oplus 0$ leads to $AT + TA=-\omega E_{12}\oplus 0 $, which yields  $\operatorname{asc}(AT + TA)=2$ and $\operatorname{asc}(AT + TA+\omega T)=1$.
	If $\eta_1+ \eta _3 \neq 0$, $\eta_1\neq -\frac{\omega}{2}$ and $\eta_1=0$,
	then taking $T=E_{12}\oplus 0$ also yields a contradiction.
	If $\eta_1+ \eta _3 \neq 0$, $\eta_1\neq -\frac{\omega}{2}$, $\eta_1\not=0$ and $\eta_3=0$, then taking rank-one operator $T=(E_{13}+E_{33})\oplus 0$ gives $AT+TA=\eta_1E_{13}\oplus 0$, which yields $\operatorname{asc}(AT+TA)=2$ and $\operatorname{asc}(AT+TA+\omega T)=\operatorname{asc}(((\eta_1+\omega)E_{13}+\omega E_{33})\oplus0)=1$.
	If $\eta_1+ \eta _3 \neq 0$, $\eta_1\neq -\frac{\omega}{2}$, $\eta_1\not=0$ and $\eta_3\not=0$, then  $A =\eta  _1\operatorname{diag} (1, 1, \frac{\eta_3}{\eta_1})  \oplus A_1$ with $c=\frac{\eta_3}{\eta_1}\not\in\{0, \pm1 \}$. By taking $T=T_c\oplus 0$ with $T_c$ as stated in Lemma \ref{L3},  and using Lemma \ref{L3},  one obtains $\operatorname{asc}(AT + TA) \neq \operatorname{asc}(AT + TA+\omega T)$.

	If $\eta_1,\eta_2,\eta_3$ are distinct, then at least two of $\eta_1,\eta_2,\eta_3$ are non-zero. Without loss of generality, assume that $\eta_1 \not=0$ and $\eta_2 \not=0$. If   $\eta_1+\eta_2=0$, putting $T=E_{12}\oplus 0$ leads to $AT + TA= 0 $ and $AT + TA+\omega T=\omega E_{12} \oplus 0 $, which yields  $\operatorname{asc}(AT + TA)=1$ and $\operatorname{asc}(AT + TA+\omega T)=2$. If    $\eta_1+\eta_2 \neq 0$ and $\eta_1= -\frac{\omega}{2}$, putting $T=(E_{11}+E_{12})\oplus 0$ gives $AT + TA= (2\eta_1E_{11}+(\eta_1+\eta_2)E_{12})\oplus 0 $ and $AT + TA+\omega T=(\eta_2+\frac{\omega}{2})E_{12} \oplus 0 $, which yields  $\operatorname{asc}(AT + TA)=1$ and $\operatorname{asc}(AT + TA+\omega T)=2$.
	If    $\eta_1+\eta_2 \not= 0$ and $\eta_1 \not=-\frac{\omega}{2}$, then $A={\eta_1}\operatorname{diag}(1,\frac{\eta_2}{\eta_1}, \frac{\eta_3}{\eta_1}) \oplus A_1$.
	Taking $T=\operatorname{diag}\{T_c, c-1\} \oplus 0$ with $c=\frac{\eta_2}{\eta_1}\not\in\{0, \pm1 \}$ as stated in Lemma \ref{L2} and applying Lemma \ref{L2}, one achieves  $\operatorname{asc}(AT + TA) \neq \operatorname{asc}(AT + TA+\omega T)$.
	
	If $\eta_1=\eta_2=\eta_3$, then $A=\operatorname{diag}(\eta_1,\eta_1, \eta_1) \oplus A_1$.
	In this case, if $\sigma(A_1)=\{\eta_1\}$, then $A \in {\mathbb F}I$, a contradiction. Thus, there exists $\eta_4 \in \sigma(A_1)$ such that $\eta_4\not=\eta_1$, and then we can rewrite $A=\operatorname{diag}(\eta_1,\eta_1, \eta_4) \oplus A_2$. By the above   discussion,  such  an  algebraic operator with degree  not 2 and rank $< \infty $, or   a rank-one operator  $T$ exists.

	{\bf Case 2.} $\mathbb F=\mathbb R$.

	If $\sigma(A)$ is real, by a similar discussion to that of Case 1, there exists an algebraic operator of degree not equal to $2$ and finite rank, or a rank-one operator $T\in\mathcal B(\mathcal X)$ such that
	$\operatorname{asc}(AT+TA)\neq\operatorname{asc}(AT+TA+\omega T)$.
	If \(\sigma(A)\) is not real, then the complexification of \(A\) has a pair of conjugate eigenvalues \(a\pm ib\), where \(b>0\), and the minimal polynomial of \(A\) can be written as
	\[
	p(\lambda)=(\lambda^2-2a\lambda+(a^2+b^2))^\beta p_1(\lambda).
	\]
		Thus, \(\mathcal X\) admits the space decomposition
	$\mathcal X=\mathcal Y\oplus\ker(p_1(A))$,
	where $\mathcal Y=\ker((A^2-2aA+(a^2+b^2)I)^\beta)$.
	It is clear,  by the real
	Jordan decomposition (see \cite[pp.201--202]{RA2013}), that $\mathcal Y$ and the algebraic operator $A|_{\mathcal Y}$ of degree 2 admit   decompositions respectively $\mathcal Y=\bigoplus_{j\in\Lambda}\mathcal M_j$ and
$A|_{\mathcal Y}=\bigoplus _{j\in\Lambda} R_{t_j}(a,b)$, where $1\leq t_j\leq \beta$, $\dim \mathcal M_j=2t_j$ and 	
	\[
	R_{t_j}(a, b) =A|_{\mathcal M_j}=
	\begin{pmatrix}
		C_{a,b} & I_2 & & & \\
		& C_{a,b} & I_2 & & \\
		& & \ddots & I_2 & \\
		& & & C_{a,b}
	\end{pmatrix} \text{ with } C_{a,b} = \begin{pmatrix} a & b \\ -b & a \end{pmatrix} \text{ and }I_2 = \begin{pmatrix} 1 & 0 \\ 0 & 1 \end{pmatrix}
	\]
	is $2t_j\times 2t_j$ matrix. \if false Denote by $\mathcal M_j$ there exist finite-dimensional closed $A$-invariant
	subspaces $\mathcal M_j\subseteq\mathcal Y$ such that
	$A|_{\mathcal M_j}$ is represented by the
	Each $\mathcal M_j$  is spanned by real Jordan chains
	corresponding to $a\pm ib$, whose lengths $t_j$
	satisfy $1\leq t_j\leq\beta<\infty$.\fi
	In particular, $R_1(a,b)=C_{a,b}$.
	Let $p$ be the cardinal  number of the index set $\Lambda$.

	{\bf Subcase 2.1.}  $\max\{t_j: j\in\Lambda\}=t_s\geq2$.
	
	In this case, $A$ can be written as
	$$A = R_{t_s}(a,b)\oplus A‘’=\begin{pmatrix}
		C_{a,b}& A_{12}\\
		0 & A_{22}
	\end{pmatrix} \oplus A''.$$
	
	Let \begin{equation}	T_{\alpha, \beta} = \begin{pmatrix}
			\beta & \frac{2\alpha^2 + \beta^2}{\alpha} \\
			\frac{\beta^2}{\alpha} & \beta
		\end{pmatrix}  \text{ and } S_{\alpha, \beta} = \begin{pmatrix}
			\beta & \frac{2\beta^2 + \alpha^2}{\alpha} \\
			\frac{2\beta^2}{\alpha} & \beta
		\end{pmatrix}. \end{equation}

	If $a \neq 0$, by taking $T=\begin{pmatrix}
		T_{a,b}& 0\\
		0 & bI_{2t_s-2}
	\end{pmatrix}  \oplus 0$,  one gets $$AT+TA=\begin{pmatrix}
		C_{a,b}T_{a,b}+T_{a,b}C_{a,b}&T_{a,b}A_{12}+bA_{12} \\
		0 & 2bA_{22}
	\end{pmatrix}  \oplus 0$$
	and
	$$AT+TA+\omega T=\begin{pmatrix}
		C_{a,b}T_{a,b}+T_{a,b}C_{a,b}+\omega T_{a,b} &T_{a,b}A_{12}+bA_{12} \\
		0 & 2bA_{22}+b\omega I_{2t_s-2}
	\end{pmatrix}  \oplus 0.$$
	Then, a simple computation gives
	$$C_{a,b}T_{a,b}+T_{a,b}C_{a,b}=
	\begin{pmatrix}
		0 &4a^2 + 4b^2 \\
		0 & 0
	\end{pmatrix}$$  and $$C_{a,b}T_{a,b}+T_{a,b}C_{a,b}+\omega T_{a,b}=	\begin{pmatrix}
		\omega b & m \\
		\frac{\omega b^2}{a} & \omega b
	\end{pmatrix},$$  where $ m=\frac{\omega b^2+2\omega a^2+4a^3+4ab^2}{a}$.
	Since $A_{22}$ and $2A_{22}+\omega I_{2t_s-2} $ are invertible in  $ \mathcal M_{2t_s-2}(\mathbb R)$, by Proposition \ref{p1}(3)-(4) and Proposition \ref{D}, one has
	$\operatorname{asc}(AT + TA)=\operatorname{asc}(C_{a,b}T_{a,b}+T_{a,b}C_{a,b})=2$ and
	$\operatorname{asc}(AT + TA+\omega T)\leq1$.
	
	Similarly,	if $a =0$, by taking $T=\begin{pmatrix}
		S_{\omega,b}& 0\\
		0 & bI_{2t_s-2}
	\end{pmatrix}  \oplus 0$,  one gets
	
	$$CS_{\omega,b}+S_{\omega,b}C=	\begin{pmatrix}
		-\omega b & 2b^2 \\
		-2b^2 & -\omega b
	\end{pmatrix} $$ and
	$$CS_{\omega,b}+S_{\omega,b}C+\omega S_{\omega,b}=\begin{pmatrix}
		0& \omega^2 + 4b^2\\
		0&0
	\end{pmatrix} .$$ 	So $\operatorname{asc}(AT + TA+\omega T)=\operatorname{asc}(CS_{\omega,b}+S_{\omega,b}C)=2$ and
	$\operatorname{asc}(AT + TA)\leq1$.
	In addition, by a similar discussion to that of the proof in Lemma \ref{L1},  one can show that  the above two forms of  $T$ are  not  of degree two.

{\bf Subcase 2.2.}  $\max\{t_j: j\in\Lambda\}=1$ and $p\geq 2$.
	In this case we can write
	\[
	A=
	\begin{pmatrix}
		C_{a,b}&0\\
		0&C_{a,b}
	\end{pmatrix}
	\oplus A''
	\]
	 as $p\geq2$. A similar argument shows that one can  choose a suitable operator $T$ such that
	$\operatorname{asc}(AT+TA)\neq
	\operatorname{asc}(AT+TA+\omega T)$.

	{\bf Subcase 2.3.}    $\max\{t_j: j\in\Lambda\}=1$ and $p=1$.

	In this case, we may write
	$A=C_{a,b}\oplus A'$, where $A'=A|_{\ker(p_1(A))}$.
	If  $A'$ contains $J_s(\gamma)$ ($s \geq3$), $\operatorname{diag}\bigl\{C_{c,d}, C_{c,d}\}$, or
	$R_{l}(c,d)$ ($l \geq 2$), where $\gamma, c\pm id \in \sigma (A') $, similar to the    discussion in Subcase 2.1,  we may take a suitable  $T \in \mathcal B(\mathcal X)$ such that
	$\operatorname{asc}(AT + TA) \neq \operatorname{asc}(AT + TA+ \omega T)$
	Thus, we only need to consider $A$  of the following form:
	$$ \operatorname{diag}\bigl\{C_{a,b}, C_{c,d}\} \oplus A_1, \\ \operatorname{diag}\bigl\{C_{a,b}, \gamma \}\oplus A_1, \\ \operatorname{diag}\bigl\{C_{a,b}, J_2(\gamma) \} \oplus A_1.$$
	Similarly, using $T_{\alpha, \beta}$ and $S_{\alpha, \beta}$,	one can take suitable $T$  such that $\operatorname{asc}(AT + TA) \neq \operatorname{asc}(AT + TA+\omega T)$.
	
	Also note that, all operators
	$T$ constructed as above are  finite-rank operators, either rank-one or of algebraic degree not equal to two. Thus, combining the foregoing discussions, the proposition concerning the ascent is true.
	
	Symmetrically, an analogous argument establishes that the proposition is also valid for the descent.
\end{proof}

\section{ Ascent  and descent of rank-one  and rank-two operators}

In this section, we  discuss properties of  the ascent (descent) of rank-one operators and rank-two operators.
	
	Firstly, for the ascent of any rank two operators, we have the following result.

	\begin{prop} \label{rank2} Assume that   $A \in \mathcal B(\mathcal X)$ with $\operatorname{rank}(A)=2$. Then the following assertions hold.
		
		{\rm (1)} $\operatorname{asc}(A)  \in \{ {1,2,3} \}$.
		
		{\rm (2)}   $\operatorname{asc}(A) = 1$ if and only if   $\operatorname{rank}(A^2)=2$.
		
		{\rm (3)}  $\operatorname{asc}(A) = 2$ if and only if either ${A^2}=0$ or $\operatorname{rank}(A^2)=\operatorname{rank}(A^3)=1$.
		
		{\rm (4)}  $\operatorname{asc}(A) = 3$ if and only if  $\operatorname{rank}(A^2)=1 \text{ and }{A^3} = 0$.
	\end{prop}
	
	\begin{proof}
Assume that   $A \in \mathcal B(\mathcal X)$ with $\operatorname{rank}(A)=2$. We will prove this proposition by considering two cases.

{\bf Case 1.}   $\dim\mathcal X \geq 4$.

		Take any basis $\{\varphi_1,  \varphi_2\}$  in $\operatorname{ran} (A)$ and let $\psi_1,  \psi_2$ be vectors in $\mathcal X$ such that  $\varphi_1,  \varphi_2,\psi_1,  \psi_2$  are linearly independent and   $A\psi_j = \varphi_j $ $(j = 1,2)$. Then there exist two functionals  $f_1, f_2\in \mathcal X^*$ such that  $\langle \varphi_j, f_k \rangle = \delta_{jk}=\begin{cases}1 &{\rm if} \ j=k\\ 0&{\rm if}\  j\not=k \end{cases}$ for $j,k\in\{1,2\}$. It is clear that 
		$$A = \varphi_1\otimes g_1+\varphi_2 \otimes g_2,\ \ {\rm where} \ \ g_j = A^* f_j\ \ {\rm with}\ \ j=1,2.$$  Set
		$$\mathcal M_A = \operatorname{span} \{ \varphi_1,  \varphi_2, \psi_1, \psi_2 \}  \text{ and } \mathcal R_A = \operatorname{ker}(g_1)\bigcap   \operatorname{ker} (g_2).$$
		Then  $\mathcal X = \mathcal M_A + \mathcal R_A$ and any vector $x \in \mathcal X$ can be written as  $x = u + v$, where
		$$u = x - \sum_{j=1}^2 \langle x, g_j \rangle \psi_j \in \mathcal R_A  \text{ and }  v = \sum_{j=1}^2 \langle x, g_j \rangle \psi_j \in \mathcal M_A.$$
		Denote by   $\mathcal S_A = \mathcal M_A \cap \mathcal R_A$. Since $\dim \mathcal S_A < \infty$, there exists some subspace $\mathcal N_A$ such that $\mathcal R_A = \mathcal{S}_A \oplus \mathcal N_A$.  It is clear that $\mathcal X = \mathcal M_A \oplus \mathcal N_A$ with  $\dim \mathcal M_A =4$. Under the  space decomposition $\mathcal X = \mathcal M _A \oplus \mathcal N_A$, $A$ can be written as $A =  A_1 \oplus 0$ with $A_1\in \mathcal M_4(\mathbb F)$ and $\operatorname{rank}(A_1)=2$.
		Obviously,  $A_1$ is similar to one of the following forms:
$$\begin{cases}
	B_1=\operatorname{diag}(\lambda_1, \lambda_2,0,0),  &
	B_2= \operatorname{diag}\{\lambda_1, J_2(0),0\}, \\
	B_3=\operatorname{diag}\{J_2(0), J_2(0) \},  &  B_4= \operatorname{diag}\{J_2(\lambda_1), 0,0 \},\\
	B_5=\operatorname{diag}\{J_3(0),0 \},& B_6=\operatorname{diag} \{ C_{a,b}, 0,0 \},
\end{cases}$$
where $ b, \lambda_1,\lambda_2$ are  nonzero numbers.

Note that
		\begin{equation}\label{similar}
\operatorname{asc}(B_1)=\operatorname{asc}(B_4)=\operatorname{asc}(B_6)=1,\ \
		\operatorname{asc}(B_2)=\operatorname{asc}(B_3)=2 \text{ and }
		\operatorname{asc}(B_5)=3.
		\end{equation}
So, by Proposition \ref{p1}(2) and (4),
$\operatorname{asc}(A) \in \{1, 2,3 \}$.
		
Also note that 	
	\begin{equation}\label{rank}
	\operatorname{rank}(B_1^2)=\operatorname{rank}(B_4^2)=\operatorname{rank}(B_6^2)=2,\ \
\operatorname{rank}(B_2^2)=\operatorname{rank}(B_5^2)=1  \text{ and }
B_3^2=0.
\end{equation}
It is easily seen from Eqs.\eqref{similar}-\eqref{rank} that  $\operatorname{asc}(A)=1\Leftrightarrow\operatorname{rank}(A^2)=2$, that is, the statement (2) holds.

Next,   it is clear that
\begin{equation}\label{asc 2}\operatorname{rank}(B_1^2)=\operatorname{rank}(B_4^2)=\operatorname{rank}(B_6^2)=2,\ \
\operatorname{rank}(B_2^2)=\operatorname{rank}(B_2^3)=1 \text{ and }
B_3^2=B_5^3=0.
\end{equation}
As $\operatorname{asc}(A)=2$ if and only if $A_1$ is similar to either $B_2$ or $B_3$, it follows from Eq.\eqref{asc 2} that $\operatorname{asc}(A)=2\Leftrightarrow$
either $A^2=0$ or $\operatorname{rank}(A^2)=\operatorname{rank}(A^3)=1$, and so the statement (3) holds.
		
Finally,  by Eq.\eqref{similar},   $\operatorname{asc}(A)=3$ if and only if  $A_1$ is similar to  $B_5$.   As
		$$\operatorname{rank}(B_1^3)=\operatorname{rank}(B_4^3)=\operatorname{rank}(B_6^2)=2,\ \
		\operatorname{rank}(B_2^3)=\operatorname{rank}(B_5^2)=1 \text{ and }
		B_3^2=B_5^3=0,$$
it is easily checked that $\operatorname{asc}(A)=3\Leftrightarrow\operatorname{rank}(A^2)=1$ and $A^3=0$. 	That is, the statement (4) is true.

 {\bf Case 2.}   $\dim\mathcal{X}=3$.

 As rank$(A)=2$, \if false Assume that $\lambda_1, \lambda_2$ are nonzero spectra of $A$. \fi
 $A$ is similar to one of the following forms:
$$\operatorname{diag} \{ C_{a,b}, 0 \},\\  \operatorname{diag}(\lambda_1, \lambda_2,0), \\
\operatorname{diag}\{\lambda_1, J_2(0)\}, \\
\operatorname{diag}\{J_2(\lambda_1), 0 \},\\
J_3(0)$$
with nonzero $b, \lambda_1,\lambda_2$.
A similar analysis to that of Case 1 shows that assertions (1)-(4) still hold.
\end{proof}

By Proposition \ref{rank2},  we can give
characterizations about the ascent of  Jordan product of any rank one operator and any bounded linear operator.
 Note that, rank($Ax\otimes f+x\otimes A^*f)\leq 2$.

\begin{prop}\label{main1}
	Assume that $A \in \mathcal B(\mathcal X)$ and $x \otimes f \in \mathbb F^* \mathcal P_1(\mathcal X)$. The following statements are true.
	
	{\rm (1)} If $\operatorname{rank}(Ax \otimes f + x \otimes A^*f) = 1$, then  $$\operatorname{asc}(Ax \otimes f + x \otimes A^*f) = 1\Leftrightarrow Ax \neq 0\ {\rm and}\ A^*f \neq 0$$and $$\operatorname{asc}(Ax \otimes f + x \otimes A^*f) = 2\Leftrightarrow \ {\rm either}\ Ax=0\ {\rm or}\ A^*f=0.$$

	{\rm (2)} If $\operatorname{rank}(Ax \otimes f + x \otimes A^*f) = 2$, then $$\operatorname{asc}(Ax \otimes f + x \otimes A^*f) = 1\Leftrightarrow\operatorname{rank}((Ax \otimes f + x \otimes A^*f)^2)=2,$$
	$$\operatorname{asc}(Ax \otimes f + x \otimes A^*f) = 2\Leftrightarrow f(A^2x)=\alpha f(Ax)=\alpha^2 f(x)\ {\rm for\ some}\ \alpha \in \mathbb{F}^*$$
	and  $$\operatorname{asc}(Ax \otimes f + x \otimes A^*f) = 3\Leftrightarrow f(Ax)=f(A^2x)=0.$$
	
\end{prop}

\begin{proof}
For $A \in \mathcal B(\mathcal X)$ and any $x \otimes f \in \mathbb F^* \mathcal P_1(\mathcal X)$, a direct calculation gives
\begin{align}\label{A2}
		(Ax \otimes f + x \otimes A^*f)^2
		=&(f(Ax)Ax +f(A^2x)x )\otimes f + (f(x)Ax +f(Ax)x) \otimes A^*f
	\end{align}
and \begin{align}\label{A3}
		(Ax \otimes f + x \otimes A^*f)^3
		=& (f^2(Ax)+ f(x)f(A^2x))Ax \otimes f + 2f(A^2x)f(Ax)x \otimes f \\ \nonumber
		&+ 2f(Ax)f(x)Ax \otimes A^*f +( f(A^2x)f(x)+ f^2(Ax))x \otimes A^*f.
	\end{align}
	
{\bf Case 1.}   $\operatorname{rank}(Ax \otimes f + x \otimes A^*f) = 1$.

In this case,  $Ax=\lambda x$ or $A^*f=\mu f$ for some $\lambda,\mu \in\mathbb{F} $ with $\lambda+\mu\neq 0$. By Proposition \ref{p1}(7),  one gets $\operatorname{asc}(Ax \otimes f + x \otimes A^*f) \in\{1,2\}$.

	If $Ax=\lambda x$, then $Ax \otimes f + x \otimes A^*f=x \otimes (\lambda f+A^*f)$. By
		Proposition \ref{p1}(7), one has
		$\operatorname{asc}(Ax \otimes f + x \otimes A^*f) = 2$ if and only if   $\langle x, \lambda f+A^*f\rangle =2\lambda\langle x,f\rangle=0$. As  $\langle x,f\rangle \neq 0$, we see that  $\operatorname{asc}(Ax \otimes f + x \otimes A^*f) = 2$ if and only if   $\lambda=0$.
			
			Similarly, one can show that, if $A^*f=\mu f$, then $\operatorname{asc}(Ax \otimes f + x \otimes A^*f) = 2$ if and only if   $\mu=0$.
			
			Thus, we have proved that $\operatorname{asc}(Ax \otimes f + x \otimes A^*f) = 2$ if and only if   $Ax=0$ or $A^*f=0$. Hence $\operatorname{asc}(Ax \otimes f + x \otimes A^*f) = 1\Leftrightarrow$  $Ax\not=0$ and $A^*f\not=0$.

{\bf Case 2.}   $\operatorname{rank}(Ax \otimes f + x \otimes A^*f) = 2$.

 By  Proposition \ref{rank2}, $\operatorname{asc}(Ax \otimes f + x \otimes A^*f) \in\{1,2,3\}$, and moreover,
\begin{equation}
\operatorname{asc}(Ax \otimes f + x \otimes A^*f) = 1\Leftrightarrow{\rm rank}((Ax \otimes f + x \otimes A^*f)^2)=2.
\end{equation}
	
Note that, in this case, both $\{Ax, x\}$ and  $\{A^*f, f\}$ are linearly independent sets.
So Eq.\eqref{A2} implies
$$(Ax \otimes f + x \otimes A^*f)^2\not =0.$$
It follows that
 \begin{equation}\label{A7}
 	\operatorname{asc}(Ax \otimes f + x \otimes A^*f)\in\{2,3\}\Leftrightarrow{\rm rank}((Ax \otimes f + x \otimes A^*f)^2)=1.
 \end{equation}

  If $\operatorname{rank}(Ax \otimes f + x \otimes A^*f)^2=1$, by Eq.\eqref{A2},  
there exists some $\alpha \in \mathbb{F}$ such that
	$$f(Ax)Ax+f(A^2x)x=\alpha f(x)Ax + \alpha f(Ax)x, $$which is equivalent to
\begin{equation}\label{A8}f(A^2x)=\alpha f(Ax) \text{ and }  f(Ax)=\alpha f(x)
	\end{equation}
as $Ax$ and $x$ are linearly independent.
Combining Eq.\eqref{A8} and Eq.\eqref{A3} gives
	$$  (Ax \otimes f + x \otimes A^*f)^3=2 \alpha f^2(x)(Ax + \alpha x) \otimes( A^*f+\alpha f),$$
and so  $\operatorname{rank}((Ax \otimes f + x \otimes A^*f)^3) \leq1$.

\if false If $\alpha=0$, then  $(Ax \otimes f + x \otimes A^*f)^3 = 0$ by Eq.\eqref{A3}.  It follows from Proposition \ref{rank2} (4) that
 $\operatorname{asc}(Ax \otimes f + x \otimes A^*f) = 3$.

 If $\alpha\not=0$, then
$f(A^2x)=\alpha f(Ax)=\alpha^2 f(x)$, and thus
	$$(Ax \otimes f + x \otimes A^*f)^2=f(x)(Ax + \alpha x) \otimes( A^*f+\alpha f) $$
	and
	$$  (Ax \otimes f + x \otimes A^*f)^3=2 \alpha f^2(x)(Ax + \alpha x) \otimes( A^*f+\alpha f).$$\fi
	
Now, by  Proposition \ref{rank2}(3)-(4), Eq.\eqref{A7} implies that  	$\operatorname{asc}(Ax \otimes f + x \otimes A^*f) = 2\Leftrightarrow f(A^2x)=\alpha f(Ax)=\alpha^2 f(x)\ {\rm for\ some}\ \alpha \in \mathbb{F}^*$,
and that  $\operatorname{asc}(Ax \otimes f + x \otimes A^*f) = 3\Leftrightarrow f(Ax)=f(A^2x)=0.$

The proof is finished.
	\if false 	(5) Assume that $\operatorname{asc}(Ax \otimes f + x \otimes A^*f) = 3$.   By (4) in Proposition \ref{rank2}, we get $$(Ax \otimes f + x \otimes A^*f)^2 \neq 0 \text{ and }  (Ax \otimes f + x \otimes A^*f)^3=0.$$ Then we have
	$$\begin{aligned}
		(Ax \otimes f + x \otimes A^*f)^3x
		=& (f^2(Ax)+ f(x)f(A^2x))f(x)Ax + 2f(A^2x)f(Ax)f(x)x  \\
		+& 2f^2(Ax)f(x)Ax +( f^2(Ax)f(x)+ f(A^2x)f(x))f(Ax)x
		=0.
	\end{aligned}$$
	Since $x$ and $Ax$ are linearly independent, we get $$3f^2(Ax)+f(x)f(A^2x)=0 \text{ and } f^3(Ax)+3f(A^2x)f(Ax)f(x)=0.$$ Thus we may conclude that  $f(A^2x)=f(Ax)=0$.
	
	For the other hand, since $f(A^2x)=f(Ax)=0$ and $f(x) \neq 0$, by Eq.(2.1) and Eq.(2.2) we can get $$(Ax \otimes f + x \otimes A^*f)^2 \neq 0 \text{ and } (Ax \otimes f + x \otimes A^*f)^3=0.$$  By (6) in Proposition \ref{p1}, we have $\operatorname{asc}(Ax \otimes f + x \otimes A^*f)=3$.\fi
\end{proof}

\begin{prop}\label{main2}
	\label{main2}  Assume that   $A \in {\mathcal B}(\mathcal{X})$ and $ x \otimes f \in \mathcal N_1(\mathcal X)$. Then the following statements hold.
	
	{\rm (1)} $\operatorname{asc}(Ax \otimes f + x \otimes A^*f) =1\Leftrightarrow$ either $Ax \otimes f + x \otimes A^*f=0$ or $ f(Ax) \neq 0 $.
	
	{\rm (2)} $\operatorname{asc}(Ax \otimes f + x \otimes A^*f) =2\Leftrightarrow Ax \otimes f + x \otimes A^*f \neq 0$ and $f(Ax)=f(A^2 x)=0$.
	
	{\rm (3)} $\operatorname{asc}(Ax \otimes f + x \otimes A^*f) =3\Leftrightarrow f(Ax)=0$ and $f(A^2 x) \neq 0$.
	
	{\rm (4)} $\operatorname{asc}(Ax \otimes f + x \otimes A^*f) \in \{2,3\}\Leftrightarrow Ax \otimes f + x \otimes A^*f \neq 0 $ and $f(Ax)=0$.
\end{prop}

\begin{proof}
Assume that $A \in \mathcal B(\mathcal X)$ and $ x \otimes f \in \mathcal N_1(\mathcal X)$.  By Proposition \ref{rank2}, one has $\operatorname{asc}(Ax \otimes f + x \otimes A^*f)\in\{1,2,3\}$. 	In fact, we only need to prove  that the statements (2) and (3) hold.


 (2):  If $Ax \otimes f + x \otimes A^*f \neq 0$ and $f(Ax)=f(A^2x)=0 $, then  $(Ax \otimes f + x \otimes A^*f )^2=0$. It follows from Proposition \ref{p1}(7) and Proposition \ref{rank2}(3) that  $\operatorname{asc}(Ax \otimes f + x \otimes A^*f) =2 $.

Conversely,
assume  that  $\operatorname{asc}(Ax \otimes f + x \otimes A^*f) =2$.  Obviously,  $Ax \otimes f + x \otimes A^*f \neq 0$.

If   $\operatorname{rank}(Ax \otimes f + x \otimes A^*f)=1$, by Proposition  \ref{p1}(7), $Ax \otimes f + x \otimes A^*f \in \mathcal N_1(\mathcal X)$,  and  so
$$f(Ax)=f(A^2x)=0.$$

 If $\operatorname{rank}(Ax \otimes f + x \otimes A^*f)=2$, by Proposition \ref{rank2}(3),  one has either
$(Ax \otimes f + x \otimes A^*f)^2=0$ or
$\operatorname{rank}((Ax \otimes f + x \otimes A^*f)^2) =\operatorname{rank}((Ax \otimes f + x \otimes A^*f)^3)=1$.  We claim that $(Ax \otimes f + x \otimes A^*f)^2=0$. Otherwise,
  $\operatorname{rank}((Ax \otimes f + x \otimes A^*f)^2) =\operatorname{rank}((Ax \otimes f + x \otimes A^*f)^3)=1$.  Then $f(Ax)\not=0$ as $(Ax \otimes f + x \otimes A^*f)^3 \not= 0$.
Since  $\operatorname{asc}(Ax \otimes f + x \otimes A^*f) =\operatorname{desc}(Ax \otimes f + x \otimes A^*f) =2$,  $(Ax \otimes f + x \otimes A^*f)^2 $ and  $(Ax \otimes f + x \otimes A^*f)^3 $ have the same kernel and range. Thus there exists some $\lambda \in  \mathbb F^*$ such that
\begin{equation}\label{23}
	(Ax \otimes f + x \otimes A^*f)^3 = \lambda(Ax \otimes f + x \otimes A^*f)^2.
	\end{equation} 
A direct calculation gives $\lambda f^2(Ax)x=f^3(Ax)x$, and so
 $f(Ax)=\lambda$. This and Eq.\eqref{23} yield   $f(A^2x)=0$. It follows that $(Ax \otimes f + x \otimes A^*f)^2 = \lambda(Ax \otimes f + x \otimes A^*f)$, which is impossible. Hence
	$(Ax \otimes f + x \otimes A^*f)^2=0$.
This  implies
$0=(Ax \otimes f + x \otimes A^*f)^2x=f^2(Ax)x$, that is, $f(Ax)=0$. It follows from $(Ax \otimes f + x \otimes A^*f)^2=0$ again that   $f(A^2 x) = 0$.
	
 (3):  	If  $f(Ax)=0$ and $f(A^2 x) \neq 0$, a direct calculation gives   $$(Ax \otimes f + x \otimes A^*f)^2 \neq 0\text{ and } (Ax \otimes f + x \otimes A^*f)^3=0.$$  By  Proposition \ref{p1}(6),  $\operatorname{asc}(Ax \otimes f + x \otimes A^*f) =3 $.

Now, assume that  $\operatorname{asc}(Ax \otimes f + x \otimes A^*f) =3 $.  By Proposition \ref{p1}(7) and Proposition \ref{rank2}(4),  one has
$$\operatorname{rank}(Ax \otimes f + x \otimes A^*f) =2,\ \ (Ax \otimes f + x \otimes A^*f)^2 \neq 0 \text{ and }  (Ax \otimes f + x \otimes A^*f)^3=0.$$
It is easily checked that
 $f(Ax)=0$ and  $f(A^2 x) \neq 0$.
	
The proof is completed.
	\end{proof}
	
	By  Propositions \ref{rank2}-\ref{main2}, for  the ascent of Jordan product of any rank-one operators, we have the following result.
	
\begin{prop}\label{two1}
	Assume that   $A, B \in \mathcal F_1(\mathcal X)$. Then  $\operatorname{asc}(AB + BA) \in \{1,2\} $.
\end{prop}

\begin{proof}
	Let  $A = x \otimes f $ and  $B = y \otimes g $ with $x,y\in\mathcal X$, $f,g\in\mathcal X^*$. It is a direct calculation that $$ (AB + BA)^2=f(y)g(x)[x \otimes(f(y)g+g(y)f)+y \otimes(f(x)g+g(x)f)].$$
	
	If $\operatorname{rank}(AB+BA) \leq1$, then $\operatorname{asc}(AB + BA) \in \{1,2\}$ by Proposition  \ref{p1}(7) and $\operatorname{asc}(0)=1$.
	
	If $\operatorname{rank}(AB+BA)=2$, then  $f(y), g(x)\not=0$ and  both $x,y$ and $f,g$ are linearly independent. By Proposition \ref{rank2}, one gets $\operatorname{asc}(AB + BA) \in \{1,2,3\}$ and $\operatorname{rank}((AB+BA)^2)=1$ if $\operatorname{asc}(AB + BA)=3$.
	
	If  $\operatorname{asc}(AB + BA)=3$,  then there exists $m \in \mathbb F^* $ such that $f(y)g+g(y)f=m(f(x)g+g(x)f)$, which implies $f(x) \not=0$ and $g(y) \not=0$. However, by  Propositions \ref{main1}-\ref{main2}, one has $f(Bx)=0$, that is,  $f(y)g(x)=0$, a contradiction. Hence $\operatorname{asc}(AB + BA) \in \{1,2\}$.
	
	The proof is finished.
\end{proof}

\section{ Characterizing operators  by  ascent (descent) of Jordan products}

In this section, we   characterize  some classes of bounded linear operators by the ascent (descent) of Jordan products.

The following two results  characterize the scalar operators  by the ascent of Jordan products.

\begin{prop} \label{0}
	Assume that  $A \in \mathcal B(\mathcal X)$. Then  the following statements  are equivalent.
	
	{\rm (1)} $A = 0$.
	
	{\rm (2)} $\operatorname{asc}(AT + TA)= 1$ for all $T \in \mathcal F_3(\mathcal X)$.
	
	{\rm (3)} $\operatorname{asc}(AN + NA)= 1$ for all $N \in \mathcal N_1(\mathcal X)$.
\end{prop}

\begin{proof}
	(1) $\Rightarrow$ (2): It is clear as  $\operatorname{asc}(0) =\operatorname{desc}(0)= 1$.
	
	(2) $\Rightarrow$ (3) is obvious.

	(3) $\Rightarrow$ (1): Assume that $A\not=0$.  As $\operatorname{asc}(AN + NA) = 1$ holds for all $N \in \mathcal N_1(\mathcal X)$, by Proposition \ref{p1}(7), $A\not\in \mathbb F I$. So there exists a nonzero vector $x_0\in\mathcal X$ such that $x_0$ and $Ax_0$ are linearly independent.  Choose $f_0 \in \mathcal{X^*} $ such that $f_0(x_0)=f_0(Ax_0)=0$, and take $N=x_0\otimes f_0$. Clearly,  $AN + NA \neq 0$. It follows from Proposition \ref{main2} that  $\operatorname{asc}(AN + NA)\in \{2,3\}$, a contradiction.  Hence  $A=0$.
\end{proof}

 Before the statement of next result, we recall a relation between rank-one operators.
Let $x \otimes f, y \otimes g\in \mathcal B(\mathcal X)$ be  rank-one operators.
We say that $x\otimes f$ and $y\otimes g$ are adjacent, denote  by $x \otimes f \sim y \otimes g$,  if rank$(x \otimes f - y \otimes g)\leq 1$. Clearly, $x \otimes f \sim y \otimes g$ if and only if  either $x$ and $y$ are linearly dependent or $f$ and $g$ are linearly dependent 
(\cite{RH2023}), and in turn, if and only if  $x \otimes f $ and $y\otimes g$ have either the same range or the same kernel.  Note that $\sim $ is not an equivalence relation because $x\otimes f\sim y\otimes g$ and $y\otimes g\sim z\otimes h$ do not imply $x\otimes f\sim z\otimes h$.

\begin{prop}\label{I}
	Assume that
	$A\in\mathcal B(\mathcal X)$.
	Then the following statements are equivalent.

	{\rm (1)} $A\in\mathbb F^*I$.
	
	{\rm (2)} $\operatorname{asc}(AT+TA)=\operatorname{asc}(T)$
	for every $T\in\mathcal B(\mathcal X)$.
	
	{\rm (3)} $\operatorname{asc}(AN+NA)=2$
	for every $N\in\mathcal N_1(\mathcal X)$.
	
    {\rm (4)}	$\operatorname{asc}(AT+TA)=0$
	for every injective $T\in\mathcal B(\mathcal X)$.
	
	{\rm (5)} $\operatorname{asc}\bigl((AT+TA)^2\bigr)
	=\operatorname{asc}(T^2)$
	for every $T\in\mathcal F_3(\mathcal X)$.
	
	{\rm (6)}  $\operatorname{asc}(AB+BA)=\operatorname{asc}(AC+CA) \Leftrightarrow\operatorname{asc}(BC+CB)=1$ for every pair of linearly independent rank-one
	operators $B,C\in\mathcal F_1(\mathcal X)$
	satisfying $B\sim C$.
	
\end{prop}

\begin{proof}

``(1) $\Rightarrow$ (2)":  This is obvious.

``(2) $\Rightarrow$ (3), (4), and (5)": These follow from
$\operatorname{asc}(N)=2$ for every
$N\in\mathcal N_1(\mathcal X)$,
$\operatorname{asc}(T)=0$ for every injective
$T\in\mathcal B(\mathcal X)$, and
Proposition \ref{ascent power}, respectively.


	{ ``(1) $\Rightarrow$ (6)":
Let $A=\lambda I$ with $\lambda\neq0$,
and let $B,C\in\mathcal F_1(\mathcal X)$
be linearly independent rank-one
operators satisfying $B\sim C$.

If $B$ and $C$ have the same range,  we can
write $B=x\otimes f$ and
$C=x\otimes g$ with
$f,g\in\mathcal X^*$
linearly independent.
then
$$BC+CB=x\otimes(f(x)g+g(x)f).$$
If $f(x)g(x)\neq0$, then
$$\operatorname{asc}(B)
=\operatorname{asc}(C)
=\operatorname{asc}(BC+CB)=1;$$
if $f(x)=g(x)=0$, then
$$\operatorname{asc}(B)
=\operatorname{asc}(C)=2\ \ {\rm and}\ \ \operatorname{asc}(BC+CB)=1;$$
if one of $f(x)$ and $g(x)$ is zero
and the other is nonzero,
then
$$\operatorname{asc}(B)\neq
\operatorname{asc}(C)\ \ {\rm and}\ \ \operatorname{asc}(BC+CB)=2.$$
These imply that (6) holds.

If $B$ and $C$ have the same kernel,
a similar argument to that of the above yields the same conclusion.}

``(3) $\Rightarrow$  (1)":
If there exists a nonzero vector $x_0\in\mathcal X$ such that $x_0$ and $Ax_0$ are linearly independent,
then we may choose $f_0 \in \mathcal X^* $ such that $f_0(x_0)=0$ and $f_0(Ax_0)=1$. Let $N=x_0\otimes f_0$.  By  Proposition \ref{main2}(1),
$\operatorname{asc}(AN + NA)=1$, a contradiction.  So, we have proved that, for any $x  \in \mathcal X$,   $x$ and $Ax$ are linearly dependent.
Combining this and Proposition \ref{0} gives $A\in \mathbb F^*I$.
	
	 ``(4) $\Rightarrow$ (1)":
	Obviously, $A\not=0$. If $A\not\in\mathbb F^*I$, then there exists a nonzero vector $x_0 \in \mathcal{X} $ such that $x_0$, $Ax_0$ are linearly independent. Under the space decomposition  $\mathcal X= \operatorname{span}\{x_0,Ax_0\} \oplus \mathcal Z$ for some subspace $\mathcal Z$, $A$ has the matrix representation
	$A=\begin{pmatrix}
		0&a_{12}&A_{13}\\
		1&a_{22}&A_{23}\\
		0&A_{32}&A_{33}
	\end{pmatrix} .$
	Take an invertible operator $T=\begin{pmatrix}
		1&0&0\\
		0&-1&0\\
		0&0&-I_z
	\end{pmatrix} \in\mathcal B(\mathcal X)$.
	As
	$AT + TA = \begin{pmatrix}
		0&0&0\\
		0&-2a_{22}&-2A_{23}\\
		0&-2A_{32}&-2A_{33}
	\end{pmatrix},$
	it is obvious that $AT+TA$ is  neither injective nor surjective.  So $\operatorname{asc}(AT + TA) \neq 0$, a contradiction. Hence  $A \in \mathbb F^*I$.

	{ ``(5) $\Rightarrow$  (1)":
We first prove that $A\not=0$. Otherwise,  choose $T\in\mathcal F_3(\mathcal X)$
such that $T^2\neq0$ and $T^3=0$.
Then $\operatorname{asc}(T^2)=2$, whereas
$\operatorname{asc}((AT+TA)^2)=\operatorname{asc}(0)=1$,
a contradiction.

Next, we show that $A$ is algebraic of degree
	 at most $2$. Suppose, on the contrary, that
	 $x_0,Ax_0,A^2x_0$ are linearly independent
	 for some $x_0\in\mathcal X$.
	 Choose $f_0\in\mathcal X^*$ such that
	 $f_0(x_0)=f_0(Ax_0)=0$ and
	 $f_0(A^2x_0)=c\neq0$.
	 Set $T=x_0\otimes f_0$.
	 Then $T^2=0$, $(AT+TA)^2=cT\neq0$  and $(AT+TA)^4=0$.
	 Consequently,
	 $\operatorname{asc}(T^2)=1$ and
	 $\operatorname{asc}((AT+TA)^2)=2$,
	 contradicting with (5).
	 Thus $x,Ax,A^2x$ are linearly dependent
	 for every $x\in\mathcal X$.
	 By Kaplansky's theorem \cite{HR1973},
	 $A$ is algebraic of degree at most $2$.
	
Finally, we  prove that $A\in\mathbb F^* I$.
Suppose, on the contrary, that
	 $A\notin\mathbb F^* I$.
	 Let $p(t)$ be a monic quadratic polynomial
	 such that $p(A)=0$.

	 {\bf Case 1.}
	 $p(t)=(t-\lambda)(t-\mu)$ with  $\lambda\neq\mu$.
	
	In this case,
	 $\mathcal X=\ker(A-\lambda I)
	 \oplus\ker(A-\mu I)$.
	 Since $\dim\mathcal X\geq3$,
	 we may assume, without loss of generality,
	 that $\dim\ker(A-\lambda I)\geq2$.
By a suitable similarity transformation, $A$ can be written as
	\[
	A=
	\begin{pmatrix}
		A_{11}&A_{12}\\
		0&A_{22}
	\end{pmatrix}
	\text{ with }
	A_{11}=\operatorname{diag}(\lambda ,\lambda ,\mu).\]
Take \[T=T_1\oplus0 \text{ with }
	T_1=
	\begin{pmatrix}
		-1&-1&-1\\
		-1&0&-1\\
		1&1&1
	\end{pmatrix}.
	\]
	Then $\operatorname{rank}(T)=2$,
	$T^2\neq0$ and $T^3=0$.
	So  $\operatorname{asc}(T^2)=2$.
Note that
	\[
	AT+TA=
	\begin{pmatrix}
		A_{11}T_1+T_1A_{11}&T_1A_{12}\\
		0&0
	\end{pmatrix} \text{ with }
	A_{11}T_1+T_1A_{11}=
	\begin{pmatrix}
		-2\lambda&-2\lambda&-(\lambda+\mu)\\
		-2\lambda&0&-(\lambda+\mu)\\
		\lambda+\mu&\lambda+\mu&2\mu
	\end{pmatrix},
	\]
and
 $$\det (A_{11}T_1+T_1A_{11})=2\lambda(\lambda-\mu)^2.$$
If $\lambda\neq0$,	then $A_{11}T_1+T_1A_{11}$ is invertible,
	and so $\operatorname{asc}(A_{11}T_1+T_1A_{11})=0$.
	If $\lambda=0$, then $\mu\neq0$ and
	$\operatorname{rank}(A_{11}T_1+T_1A_{11})
	=\operatorname{rank}((A_{11}T_1+T_1A_{11})^2)=2$,
	and so $\operatorname{asc}(A_{11}T_1+T_1A_{11})=1$.
Hence  $\operatorname{asc}(A_{11}T_1+T_1A_{11})\leq1$.
It follows from  Proposition \ref{upper triangular matrix} that
	$\operatorname{asc}(AT+TA)\leq2$,
	and consequently
	$\operatorname{asc}((AT+TA)^2)\leq1$,
a 	contradiction.

	{\bf Case 2.}
	$p(t)=(t-\lambda)^2$.

With respect to a suitable decomposition of
$\mathcal X$, we may write
\[
	A=
	\begin{pmatrix}
		A_{11}&A_{12}\\
		0&A_{22}
	\end{pmatrix}
	\text{ with }
A_{11}=\operatorname{diag}\{J_2(\lambda),\lambda\}.
\]

	Take
	\[
	T=T_1\oplus0
	\text{ with }
	T_1=
	\begin{pmatrix}
		0&0&0\\
		-1&-1&-1\\
		-1&1&1
	\end{pmatrix}.
	\]
	Then $\operatorname{rank}(T)=2$,
	$T^2\neq0$ and $T^3=0$.
So $\operatorname{asc}(T^2)=2$.
In this case, 	\[
AT+TA=
\begin{pmatrix}
	A_{11}T_1+T_1A_{11}&T_1A_{12}\\
	0&0
\end{pmatrix} \ \ {\rm with}\ \
	A_{11}T_1+T_1A_{11}=
	\begin{pmatrix}
		-1&-1&-1\\
		-2\lambda&-2\lambda-1&-2\lambda\\
		-2\lambda&2\lambda-1&2\lambda
	\end{pmatrix}.
	\]
It is a direct calculation that
	$\det(A_{11}T_1+T_1A_{11})=4\lambda$.
By a similar argument to that of Case 1, one can show that 
 $\operatorname{asc}(A_{11}T_1+T_1A_{11})\leq1$. Hence 
	$\operatorname{asc}(AT+TA)\leq2$ and
	$\operatorname{asc}((AT+TA)^2)\leq1$,
a contradiction again.

	{\bf Case 3.}
	$\mathbb F=\mathbb R$ and
	$p(t)=(t-a)^2+b^2$
	for some $a,b\in\mathbb R$ with $b\neq0$.
	
In this case, we may write
\[
	A=
	\begin{pmatrix}
		A_{11}&A_{12}\\
		0&A_{22}
	\end{pmatrix}
	\text{ with }
	A_1=\operatorname{diag}\{C_{a,b},C_{a,b}\}.
	\]
	Take
	\[
	T=T_1\oplus0
	\text{ with }
	T_1=
	\begin{pmatrix}
		0&0&0&0\\
		1&0&0&0\\
		-3&1&0&0\\
		1&0&1&0
	\end{pmatrix}.
	\]
	Then $\operatorname{rank}(T)=3$,
	$T^2\neq0$ and $T^4=0$.
So $\operatorname{asc}(T^2)=2$.
On the other hand, as
	\[
	A_{11}T_1+T_1A_{11}=
	\begin{pmatrix}
		b&0&0&0\\
		2a&b&0&0\\
		-6a&2a-3b&b&0\\
		2a+3b&0&2a&b
	\end{pmatrix}
	\]
with  $\det(A_{11}T_1+T_1A_{11})=b^4\neq0$,
	one gets $\operatorname{asc}(A_{11}T_1+T_1A_{11})=0$.
	By Proposition \ref{upper triangular matrix},
	$\operatorname{asc}(AT+TA)\leq1$,
	and consequently
	$\operatorname{asc}((AT+TA)^2)\leq1$,
a 	contradiction again.
	
Combining Cases 1-3, we see that $p(t)=\alpha + t$  and hence	$A\in\mathbb F^*I$.

``(6) $\Rightarrow$ (1)":
Firstly,  $A\neq0$.
Otherwise, if $A=0$, choose
$0\neq x\in\mathcal X$ and linearly
independent $f,g\in\mathcal X^*$
with $f(x)=0$ and $g(x)=1$.
For $B=x\otimes f$ and $C=x\otimes g$,
we have $B\sim C$,
$BC+CB=B$ and
$\operatorname{asc}(BC+CB)=2$,
while
$\operatorname{asc}(AB+BA)
=\operatorname{asc}(AC+CA)=1$,
a contradiction.

Suppose that $A$ is not a scalar operator.
Choose $x\in\mathcal X$ such that
$x,Ax$ are linearly independent.
Since $\dim\mathcal X\geq3$,
there exist nonzero $f,g\in\mathcal X^*$
such that $f(x)=0$, $f(Ax)=1$
and $g(x)=g(Ax)=0$.
Set $B=x\otimes f$ and $C=x\otimes g$.
Then $B$ and $C$ are linearly
independent rank-one operators
satisfying $B\sim C$, and
$BC+CB=0$. So
$\operatorname{asc}(BC+CB)=1$.
Since $f(Ax)=1$, Proposition \ref{main2}(1)
gives $\operatorname{asc}(AB+BA)=1$.
On the other hand, as
$CA+AC=Ax\otimes g+x\otimes A^*g\neq0$
and  $g(Ax)=0$, Proposition \ref{main2}(4)
yields $\operatorname{asc}(AC+CA)\in\{2,3\}$,
a contradiction to  (6).
Therefore $A\in\mathbb F^*I$.}
\end{proof}

 The following result gives a sufficient condition for an operator to be  algebraic of degree at most $n$ in terms of ascent of Jordan products.

\begin{prop}\label{degree at most n}
		Assume that $A\in\mathcal B(\mathcal X)$ and $n>1$ is any fixed integer.
		If
		$\sup\limits_{T\in\mathcal F_{n-1}(\mathcal X)}
		\operatorname{asc}(AT+TA)<n$,
		then $A$ is an algebraic operator of degree at most $n$.
\end{prop}

\begin{proof}
	Assume that $A$ is not an algebraic operator of degree at most $n$.
	
	Note that there exists some $x_0\in\mathcal X$ such that
	$x_0,Ax_0,\ldots,A^nx_0$ are linearly independent. Otherwise,
	if $x,Ax,\ldots,A^nx$ are linearly dependent for every
	$x\in\mathcal X$, then $A$ is locally algebraic of uniformly
	bounded degree at most $n$. By
	\cite[Kaplansky's Theorem]{HR1973}, $A$ is an algebraic operator
	of degree at most $n$, a contradiction.
	
	Thus, take
	$\{f_0,f_1,\ldots,f_n\}\subset\mathcal X^*$ such that
	$\langle A^ix_0,f_j\rangle=\delta_{ij}$,
	$i,j=0,1,\ldots,n$. Let
	$\mathcal Y=\operatorname{span}\{x_0,Ax_0,\ldots,A^nx_0\}$
	and
	$\mathcal Z=\bigcap_{i=0}^{n}\ker(f_i)$.
	Then $\mathcal X=\mathcal Y\oplus\mathcal Z$, and according to
	this space decomposition, $A$ can be represented as
	\[
	A=
	\begin{pmatrix}
		0&0&0&\cdots&0&a_1&A_1\\
		1&0&0&\cdots&0&a_2&A_2\\
		0&1&0&\cdots&0&a_3&A_3\\
		\vdots&\ddots&\ddots&\ddots&\vdots&\vdots&\vdots\\
		0&\cdots&0&1&0&a_n&A_n\\
		0&\cdots&0&0&1&a_{n+1}&A_{n+1}\\
		0&\cdots&0&0&0&B&A_{n+2}
	\end{pmatrix}.
	\]
	Taking
	$T=\operatorname{diag}(1,1,\ldots,1,0,0)\oplus0$
	with $\operatorname{rank}(T)=n-1$, a direct computation gives
	\[
	AT+TA=
	\begin{pmatrix}
		0&0&0&\cdots&0&0&a_1&A_1\\
		2&0&0&\cdots&0&0&a_2&A_2\\
		0&2&0&\cdots&0&0&a_3&A_3\\
		\vdots&\ddots&\ddots&\ddots&\vdots&\vdots&\vdots&\vdots\\
		0&\cdots&0&2&0&0&a_{n-1}&A_{n-1}\\
		0&\cdots&0&0&1&0&0&0\\
		0&\cdots&0&0&0&0&0&0\\
		0&\cdots&0&0&0&0&0&0
	\end{pmatrix}.
	\]
	By Proposition \ref{upper triangular matrix}, it is easy to see that
	$\operatorname{asc}(AT+TA)\geq n$, a contradiction. Thus $A$ is an algebraic operator of degree at	most $n$.
	
	The proof is finished.
\end{proof}

Next we give  characterizations of rank-one nilpotent operators.

\begin{prop}
{Let $A\in\mathcal F_1(\mathcal X)$. Then the following
	assertions are equivalent.
	
	{\rm (1)} $A\in\mathcal N_1(\mathcal X)$.
	
	{\rm (2)} $\operatorname{asc}(AP+PA)\leq2$ holds for all
	idempotent operators $P\in\mathcal B(\mathcal X)$.
	
	{\rm (3)} $\operatorname{asc}(AP+PA)\leq2$ holds for all
	rank-two idempotent operators $P\in\mathcal B(\mathcal X)$.}
\end{prop}

\begin{proof}  Write $A=x\otimes f\in\mathcal F_1(\mathcal X)$.
	
 {(1) $\Rightarrow$  (2):
Assume that $f(x)=0$.
For any idempotent operator  $P\in\mathcal B(\mathcal X)$,  it is obvious that $\operatorname{asc}(AP+PA)\leq3$. If
$\operatorname{asc}(AP+PA)=3$, by Proposition~\ref{main2}(3),
$f(Px)=0$ and $f(P^2x)\neq0$. However, this is impossible as $P^2=P$.
So
	$\operatorname{asc}(AP+PA)\leq2$.
	
	 (2) $\Rightarrow$ (3): This  is immediate.
	
    (3) $\Rightarrow$ (1):
If $f(x)\neq0$, we choose $0\neq u\in\ker f$ and $y\notin\operatorname{span}\{x,u\}$
	such that $f(y)\neq0$. Let  $\mathcal X=\operatorname{span}\{x,y,u\}\oplus \mathcal Y$ be a decomposition of $\mathcal X$.
	Let $P$ be a rank-two idempotent operator  defined by
	$Pu=u$, $Py=y$, $ P(x-u)=0$
	and $P|_{\mathcal Y}=0$.
	Then $Px=u$.
	Writing  $S=AP+PA=x\otimes P^*f+u\otimes f$,
	one gets
	$Sx=f(x)u$, $Su=0$ and $Sy=f(y)(x+u)$.
	Consequently,
	$S^2y=f(x)f(y)u\neq0$ and $ S^3=0$,
	which implies $\operatorname{asc}(S)=3$,
a	contradiction.
Hence $f(x)=0$, that is,
	$A\in\mathcal N_1(\mathcal X)$.}
\end{proof}

 By the above proposition, to get a complete characterization of  rank-one nilpotent operators, we have to give a characterization of rank-one operators in terms of Jordan product. This is done by the following result.

\begin{prop}\label{rank1}
	Assume that $0\neq A\in\mathcal B(\mathcal X)$. Then the following statements are equivalent.
	
	{\rm (1)} $\operatorname{rank}(A)=1$.
	
	{\rm (2)} $\operatorname{asc}(AT+TA)\in\{1,2,3\}$
	for every $T\in\mathcal B(\mathcal X)$.
	
	{\rm (3)} $\operatorname{asc}(AT+TA)\in\{1,2,3\}$
	for every $T\in\mathcal F_3(\mathcal X)$.
\end{prop}

\begin{proof}
	The implication $(1)\Rightarrow(2)$ follows
	from Propositions \ref{main1}-\ref{main2}, while
	$(2)\Rightarrow(3)$ is immediate.
	
 {$(3)\Rightarrow(1)$: Suppose that
$\operatorname{rank}(A)\geq2$. We will obtain a contradiction, and so (1) is true.

Firstly, by Proposition \ref{degree at most n} with $n=4$,  $A$ is an algebraic
	operator of degree at most $4$.
	
	
{\bf Case 1. }	$\dim\mathcal X=3$.

In this case, if  $A$ is invertible,
	taking $T=I$ gives  $AT+TA=2A$ is invertible. Thus, $\operatorname{asc}(AT+TA)=0$, a contradiction.
	If  $A$ is not invertible,
	then $\operatorname{rank}(A)=2$ and
	$\operatorname{ker} (A)=\operatorname{span}\{z\}$ for some vector $z\neq0$.
	Choose
	$0\not=u\notin\operatorname{ker} (A)\cup\operatorname{ran}(A)$
	and $f\in\mathcal X^*$ such that
	$f(z)f(Au)\neq0$.
Take $T=I+u\otimes f$.
	If $(AT+TA)x=2Ax+f(x)Au+f(Ax)u=0$ for some $x \in \mathcal X$, then
	$f(Ax)=0$ and $2Ax+f(x)Au=0$.
	Applying $f$ to the latter identity yields
	$f(x)f(Au)=0$, and so $f(x)=0$.
This implies  $Ax=0$. Hence $x\in\operatorname{span}\{z\}$.
As $f(z)\neq0$, we obtain $x=0$.
	Thus $AT+TA$ is injective and
	$\operatorname{asc}(AT+TA)=0$. This is a contradiction.
	
{\bf Case 2.}  $\dim\mathcal X\geq4$.

To obtain a contradiction with (3), it suffices to find $T\in\mathcal F_3(\mathcal X)$ such that either $AT+TA$ is invertible or there exists some $x\in\mathcal X$ satisfying
$(AT+TA)^3x\neq0$, but $(AT+TA)^4x=0$.

We will construct such $T$ by considering  four cases according to the algebraic degree of $A$.

 {\bf Case 2.1.} $A$ is algebraic of degree $1$.
	
In this case we have  $A=\lambda I$
	for some $\lambda\neq0$.
	Choose linearly independent vectors $x_1,x_2,x_3,x_4\in \mathcal X$.
	Put $\mathcal M=\operatorname{span}\{x_1,x_2,x_3,x_4\}$
	and choose a closed complement $\mathcal L$ of
	$\mathcal M$ in $\mathcal X$.
	Define $T\in\mathcal B(\mathcal X)$ by
	$Tx_1=0$, $Tx_i=x_{i-1}$ for $2\leq i\leq4$,
	and $T|_{\mathcal L}=0$.
	Then $\operatorname{rank}(T)=3$ and $AT+TA=2\lambda T$.
	Hence $$(AT+TA)^3x_4=8\lambda^3x_1\neq0\ \ {\rm and} \ \ (AT+TA)^4x_4=0.$$
	
	{\bf Case 2.2.} $A$ is algebraic of degree $2$.

	{\bf Subcase 2.2.1.} $p(t)=(t-\lambda)^2$.
	
	If $\lambda=0$, then $A^2=0$.
Choose $y_1,y_2\in\mathcal X$ such that
	$x_1=Ay_1$ and $x_2=Ay_2$ are
	linearly independent.
	Obviously,  $Ax_1=Ax_2=0$,
	and then $x_1,x_2,y_1,y_2$ are linearly independent.
	Put $\mathcal M=\operatorname{span}\{x_1,x_2,y_1,y_2\}$
	and choose a closed complement $\mathcal L$ of
	$\mathcal M$ in $\mathcal X$.
	Define $T\in\mathcal B(\mathcal X)$ by
	$Tx_1=0$, $Ty_1=y_2$, $Tx_2=y_1$,
	$Ty_2=x_1$, and $T|_{\mathcal L}=0$.
	Then $\operatorname{rank}(T)=3$ with
	$(AT+TA)y_2=y_1$, $(AT+TA)y_1=x_2$,  $(AT+TA)x_2=x_1$ and $(AT+TA)x_1=0$.
	It follows that
	$$(AT+TA)^3y_2=x_1\neq0\ \ {\rm and}\ \ (AT+TA)^4y_2=0$$
as desired.
	
If $\lambda\neq0$,
	write $A=\lambda I+N$ with $N^2=0$.
	If $\operatorname{rank}(N)\geq2$, choose
	$y_1,y_2\in\mathcal X$ such that
	 $x_1=Ny_1$ and $x_2=Ny_2$ are linearly independent.
	Then  $x_1,x_2,y_1,y_2$
	are also linearly independent.
	Put $\mathcal M=\operatorname{span}\{x_1,x_2,y_1,y_2\}$
	and choose a closed complement $\mathcal L$ of
	$\mathcal M$ in $\mathcal X$.
	Define $T\in\mathcal B(\mathcal X)$ by
	$Tx_1=0$, $Tx_2=x_1$, $Ty_1=x_2$,
	$Ty_2=y_1$, and $T|_{\mathcal L}=0$.
	Then $\operatorname{rank}(T)=3$ and
	$(AT+TA)x_1=0$, $(AT+TA)x_2=2\lambda x_1$,
	$(AT+TA)y_1=2\lambda x_2$, and
	$(AT+TA)y_2=2\lambda y_1+2x_1$.
	Hence we have
	$$(AT+TA)^3y_2=8\lambda^3x_1\neq0\ \ {\rm and}\ \ (AT+TA)^4y_2=0.$$
	If $\operatorname{rank}(N)=1$, choose
	$y\in\mathcal X$ with $x_1=Ny\neq0$.
	Since $\dim\ker N\geq3$, we can choose
	$x_2,x_3\in\ker N$ such that
	$x_1,x_2,x_3$ are linearly independent.
	As $y\notin\ker N$, the vectors
	$x_1,x_2,x_3,y$ are linearly independent.
	Put $\mathcal M=\operatorname{span}\{x_1,x_2,x_3,y\}$
	and choose a closed complement $\mathcal L$ of
	$\mathcal M$ in $\mathcal X$.
	Define $T\in\mathcal B(\mathcal X)$ by
	$Tx_1=0$, $Tx_2=x_1$, $Tx_3=x_2$,
	$Ty=x_3$, and $T|_{\mathcal L}=0$.
	Then $\operatorname{rank}(T)=3$. Note that
	$(AT+TA)x_1=0$, $(AT+TA)x_2=2\lambda x_1$,
	$(AT+TA)x_3=2\lambda x_2$ and $(AT+TA)y=2\lambda x_3$.
	Hence  $$(AT+TA)^3y=8\lambda^3x_1\neq0 \ \ {\rm and}\ \ (AT+TA)^4y=0.$$

	{\bf Subcase 2.2.2.}
	$p(t)=(t-\lambda)(t-\mu)$ with  $\lambda\neq\mu$.
	
In this case,
	$\mathcal X=\ker(A-\lambda I)\oplus\ker(A-\mu I)$.
	
Assume  first that $\lambda\mu\neq0$.
	If
	$\dim\ker(A-\nu I)\geq4$ for some $\nu\in\{\lambda,\mu\}$, we can choose linearly
	independent vectors
	$x_1,x_2,x_3,x_4\in\ker(A-\nu I)$.
	Put $\mathcal M=\operatorname{span}\{x_1,x_2,x_3,x_4\}$
	and choose a closed complement $\mathcal L$ of
	$\mathcal M$ in $\mathcal X$.
	 Then the operator $T\in\mathcal B(\mathcal X)$ defined by
	$Tx_1=0$, $Tx_i=x_{i-1}$ for $2\leq i\leq4$,
	and $T|_{\mathcal L}=0$ is the desired one as
	$\operatorname{rank}(T)=3$,
	$$(AT+TA)^3x_4=8\nu^3x_1\neq0\ \ {\rm and}\ \ (AT+TA)^4x_4=0.$$
	If $\dim\ker(A-\lambda I)\leq3$ and
	$\dim\ker(A-\mu I)\leq3$, then $\dim \mathcal X\leq 6$.
Without loss of generality, assume
	$p=\dim\ker(A-\lambda I)\geq
	q=\dim\ker(A-\mu I)$.
	Choose bases $\{x_1,\ldots,x_p\}$ and
	$\{y_1,\ldots,y_q\}$ of
	$\ker(A-\lambda I)$ and $\ker(A-\mu I)$,
	respectively.
	Since $\mathcal X=\ker(A-\lambda I)
	\oplus\ker(A-\mu I)$, these vectors form
	a basis of $\mathcal X$.
	Define $T\in\mathcal B(\mathcal X)$ by
	$Tx_i=Ty_i=x_i+y_i$ for $1\leq i\leq q$,
	and $Tx_j=x_j$ for $q<j\leq p$.
	Then $\operatorname{rank}(T)=p\leq3$.
	Moreover, for $1\leq i\leq q$, we have
	$(AT+TA)x_i=2\lambda x_i+(\lambda+\mu)y_i$
	and $(AT+TA)y_i=(\lambda+\mu)x_i+2\mu y_i$.
	It follows that
	$$2\mu (AT+TA)x_i-(\lambda+\mu)(AT+TA)y_i
	=-(\lambda-\mu)^2x_i$$
	and
	$$2\lambda (AT+TA)y_i-(\lambda+\mu)(AT+TA)x_i
	=-(\lambda-\mu)^2y_i.$$
	 Thus $x_i,y_i\in\operatorname{ran}(AT+TA)$ for $i\leq q$ and  $(AT+TA)x_j=2\lambda x_j$ for $q<j\leq p$.
	Since $\lambda\mu\neq0$ and $\lambda\not=\mu$,  $(AT+TA)$ is invertible and hence $\operatorname{asc}(AT+TA)=0\not\in\{1,2,3\}$, contradicting to (3).
	
Next, assume that  $\lambda\mu=0$.
	Without loss of generality, let
	$\lambda=0$ and $\mu\neq0$.
	Put $\mathcal X_0=\ker A$ and
	$\mathcal X_1=\ker(A-\mu I)$.
	Since $\operatorname{rank}(A)\geq2$,
	we have $\dim\mathcal X_1\geq2$.
	
	If $\dim\mathcal X_0\geq2$, choose
	linearly independent $x_1,x_2\in\mathcal X_0$
	and $y_1,y_2\in\mathcal X_1$.
	Put $\mathcal M=\operatorname{span}\{x_1,x_2,y_1,y_2\}$
	and choose a closed complement $\mathcal L$ of
	$\mathcal M$ in $\mathcal X$.
	Define $T\in\mathcal B(\mathcal X)$ by
	$Tx_1=0$, $Tx_2=\frac{1}{\mu}y_1$,
	$Ty_i=\frac{1}{\mu}x_i$ for $i=1,2$,
	and $T|_{\mathcal L}=0$.
	Then $\operatorname{rank}(T)=3$,
	$(AT+TA)y_i=x_i$ for $i=1,2$,
	$(AT+TA)x_2=y_1$ and $(AT+TA)x_1=0$.
	Clearly, for such $T$ we have
	$$(AT+TA)^3y_2=x_1\neq0\ \ {\rm and}\ \ (AT+TA)^4y_2=0.$$
	
	If $\dim\mathcal X_0=1$, then $\dim\mathcal X_1\geq3$. Take
	$0\neq z\in\mathcal X_0$ and  linearly independent
	$y_1,y_2,y_3\in\mathcal X_1$.
	Put $\mathcal M=\operatorname{span}\{z,y_1,y_2,y_3\}$
	and choose a closed complement $\mathcal L$ of
	$\mathcal M$ in $\mathcal X$.
	Define $T\in\mathcal B(\mathcal X)$ by
	$Ty_1=0$, $Ty_2=y_1$, $Tz=y_2$,
	$Ty_3=z$, and $T|_{\mathcal L}=0$.
	Then $\operatorname{rank}(T)=3$,
	$(AT+TA)y_1=0$, $(AT+TA)y_2=2\mu y_1$,
	$(AT+TA)z=\mu y_2$ and $(AT+TA)y_3=\mu z$. Now it is easily checked that
	$$(AT+TA)^3y_3=2\mu^3y_1\neq0\ \ {\rm and}\ \ (AT+TA)^4y_3=0.$$

	{\bf Subcase 2.2.3.} $\mathbb F=\mathbb R$ and $p(t)=t^2-2at+(a^2+b^2)$
	for some $a,b\in\mathbb R$ with $b\neq0$.
	
	Put $J=\frac{1}{b}(A-aI)$. Then $J^2=-I$.
	If $\dim\mathcal X\geq8$, choose
	$x_1, \ldots, x_4\in\mathcal X$ such that
	$x_1,Jx_1,\ldots,x_4,Jx_4$
	are linearly independent.
Let $\mathcal M=\operatorname{span}\{x_i,Jx_i: 1\leq i\leq4\}$
	and   $\mathcal L$ a closed complement of
	$\mathcal M$ in $\mathcal X$.
	Define $T\in\mathcal B(\mathcal X)$ by
	$Tx_1=0$, $Tx_i=x_{i-1}$ for $2\leq i\leq4$,
	$T(Jx_i)=0$ for $1\leq i\leq4$,
	and $T|_{\mathcal L}=0$.
	Then $\operatorname{rank}(T)=3$ and
	$$(AT+TA)^3(Jx_4)=b(b^2-4a^2)x_1-2ab^2Jx_1\neq0,\ \ (AT+TA)^4(Jx_4)=0.$$
	
	If $\dim\mathcal X<8$, then
	$\dim\mathcal X=2r$ for some $r\in\{2,3\}$ as $J^2=-I$.
	Choose a basis
	$\{x_1,Jx_1,\ldots,x_r,Jx_r\}$
	of $\mathcal X$ and define
	$T\in\mathcal B(\mathcal X)$ by
	$Tx_i=x_i$ and $T(Jx_i)=0$
	for $1\leq i\leq r$.
	Then $\operatorname{rank}(T)=r\leq3$,
	$(AT+TA)x_i=2ax_i+bJx_i$,
	and $(AT+TA)(Jx_i)=-bx_i$.
	Since $b\neq0$, $(AT+TA)$ is invertible. This completes the proof of Case 2.2.

 {\bf Case 2.3.} $A$ is algebraic of degree $3$.
	
	If $A$ is invertible, choose $x\in\mathcal X$
	such that $x,Ax,A^2x$ are linearly independent,
	and write $A^3x=a_0x+a_1Ax+a_2A^2x$ with  $a_0\neq0$.
	Choose $y\notin \operatorname{span}\{x,Ax,A^2x\}$.
Let $\mathcal M=\operatorname{span}\{x,Ax,A^2x,y\}$
	and $\mathcal L$ a closed complement  of
	$\mathcal M$ in $\mathcal X$.
	Define $T\in\mathcal B(\mathcal X)$ by
	$Tx=x$, $Ty=x+Ax$,
	$T(Ax)=T(A^2x)=0$ and $T|_{\mathcal L}=0$.
	Then $\operatorname{rank}(T)=2$,
	$(AT+TA)x=Ax$, $(AT+TA)(Ax)=0$, $(AT+TA)(A^2x)=a_0x$, and
	$(AT+TA)y\in A^2x+\operatorname{span}\{x,Ax\}$.
	Hence
	$$(AT+TA)^3y=a_0Ax\neq0\ \ {\rm and}\ \ (AT+TA)^4y=0.$$
	
So, in the sequel, we may assume that $A$ is not invertible.

	{\bf Subcase 2.3.1.} $p(t)=t^3$.
	
	Choose $x_1,x_2,x_3\in\mathcal X$ such that
	$Ax_1=0$, $Ax_2=x_1$ and $Ax_3=x_2$.
	Since $A^3=0$ and $\dim\mathcal X\geq4$,
	we can choose $x_4\in\ker A$ such that
	$x_1,x_2,x_3,x_4$ are linearly independent.
	Put $\mathcal M=\operatorname{span}\{x_1,x_2,x_3,x_4\}$
	and choose a closed complement $\mathcal L$ of
	$\mathcal M$ in $\mathcal X$.
	Define $T\in\mathcal B(\mathcal X)$ by
	$Tx_1=x_3$, $Tx_2=x_4$,
	$Tx_3=Tx_4=0$, and $T|_{\mathcal L}=0$.
	Then $\operatorname{rank}(T)=2$,
	$(AT+TA)x_1=x_2$, $(AT+TA)x_2=x_3$,
	$(AT+TA)x_3=x_4$ and $(AT+TA)x_4=0$.
So
	$$(AT+TA)^3x_1=x_4\neq0\ \ {\rm and}\ \ (AT+TA)^4x_1=0.$$

	{\bf Subcase 2.3.2.}
	$p(t)=t^2(t-\lambda)$ with  $\lambda\neq0$.
	
In this subcase,
	$\mathcal X=\mathcal X_0\oplus\mathcal X_1$,
	where $N=A|_{\mathcal X_0}$ satisfies
	$N^2=0$ and $N\neq0$, and
	$A|_{\mathcal X_1}=\lambda I$.
	If $\operatorname{rank}(N)\geq2$,
	choose $y_1,y_2\in\mathcal X_0$ such that
	$x_1=Ny_1$, $x_2=Ny_2$  linearly independent.
	Then
	$x_1,x_2,y_1,y_2$ are linearly independent.
	Put $\mathcal M=\operatorname{span}\{x_1,x_2,y_1,y_2\}$
	and choose a closed complement $\mathcal L$ of
	$\mathcal M$ in $\mathcal X$.
	Define $T\in\mathcal B(\mathcal X)$ by
	$Tx_1=0$, $Ty_1=y_2$, $Tx_2=y_1$,
	$Ty_2=x_1$, and $T|_{\mathcal L}=0$.
	Then $\operatorname{rank}(T)=3$,
	$(AT+TA)y_2=y_1$, $(AT+TA)y_1=x_2$,
	$(AT+TA)x_2=x_1$ and $(AT+TA)x_1=0$, which implies that
	$$(AT+TA)^3y_2=x_1\neq0\ \ {\rm and}\ \ (AT+TA)^4y_2=0.$$
	If $\operatorname{rank}(N)=1$, choose $x_1,x_2\in\mathcal X_0$
	such that $Ax_1=0$ and $Ax_2=x_1$.
	If $\dim\mathcal X_1\geq2$,
	choose linearly independent
	$y_1,y_2\in\mathcal X_1$.
	Put $\mathcal M=\operatorname{span}\{x_1,x_2,y_1,y_2\}$
	and choose a closed complement $\mathcal L$ of
	$\mathcal M$ in $\mathcal X$.
	Define $T\in\mathcal B(\mathcal X)$ by
	$Tx_1=0$, $Ty_1=x_1$, $Tx_2=y_1$,
	$Ty_2=x_2$, and $T|_{\mathcal L}=0$.
	Then $\operatorname{rank}(T)=3$ and
	$(AT+TA)x_1=0$, $(AT+TA)y_1=\lambda x_1$,
	$(AT+TA)x_2=\lambda y_1$
	and $(AT+TA)y_2=x_1+\lambda x_2$.
So
	$$(AT+TA)^3y_2=\lambda^3x_1\neq0\ \ {\rm and}\ \ (AT+TA)^4y_2=0.$$
	If $\dim\mathcal X_1=1$, then
	$\dim\mathcal X_0\geq3$.
	Choose $x_3\in\ker N$ such that
	$x_1,x_2,x_3$ linearly independent,
	and take $0\neq y\in\mathcal X_1$.
	Put $\mathcal M=\operatorname{span}\{x_1,x_2,x_3,y\}$
	and choose a closed complement $\mathcal L$ of
	$\mathcal M$ in $\mathcal X$.
	Define $T\in\mathcal B(\mathcal X)$ by
	$Tx_1=x_1+\lambda y$, $Tx_2=-y$,
	$Ty=x_3$, $Tx_3=0$, and $T|_{\mathcal L}=0$.
	Then $\operatorname{rank}(T)=3$,
	$(AT+TA)x_2=x_1$, $(AT+TA)x_1=\lambda^2y$,
	$(AT+TA)y=\lambda x_3$ and $(AT+TA)x_3=0$.
	Hence
	$$(AT+TA)^3x_2=\lambda^3x_3\neq0\ \ {\rm and}\ \ (AT+TA)^4x_2=0.$$

	{\bf Subcase 2.3.3.}
	$p(t)=tq(t)$, where $q(t)$  is monic of degree $2$ with $q(0)\neq0$.

	In this case,
	$\mathcal X=\ker A\oplus\mathcal X_1$ with  $B=A|_{\mathcal X_1}$  invertible.
	
	If $\dim\ker A\geq2$, choose linearly
	independent vectors $x_1,x_2\in\ker A$.
	Write $q(t)=t^2+at+b$ with  $b\neq0$.
	Take $y\in\mathcal X_1$ such that
	$y,By$  linearly independent.
    Since $q(B)=0$, we have
	$B^2y=-aBy-by$.
	Thus $y\in\operatorname{span}\{By,B^2y\}$.
	Moreover, $By,B^2y$ are linearly independent
as $B$ is invertible.
	Put $\mathcal M=\operatorname{span}\{x_1,x_2,By,B^2y\}$
	and choose a closed complement
	$\mathcal L$ of $\mathcal M$ in $\mathcal X$.
	Define $T\in\mathcal B(\mathcal X)$ by
	$Tx_1=0$, $Tx_2=-\frac{1}{b}(By+ay)$,
	$T(By)=x_1$, $T(B^2y)=x_2$,
	and $T|_{\mathcal L}=0$.
	Then $\operatorname{rank}(T)=3$,
	$(AT+TA)x_1=0$, $(AT+TA)x_2=y$,
	$(AT+TA)y=x_1$ and $(AT+TA)(By)=x_2$.
	Hence
	$$(AT+TA)^3(By)=x_1\neq0\ \ {\rm and}\ \ (AT+TA)^4(By)=0.$$

If  $\dim\ker A=1$, then $\dim\mathcal X_1\geq3$. Write $\ker A=\operatorname{span}\{z\}$ for some $z\in\mathcal X$. If $q(t)=(t-\lambda)(t-\mu)$
	for some $\lambda,\mu\in\mathbb F$ with $\lambda\mu \not=0$,
	we may take linearly independent vectors
	$y_1,y_2,y_3\in\mathcal X_1$ such that
	$Ay_1=\lambda_1y_1$,
	$Ay_2=\lambda_2y_2+cy_1$,
	and $Ay_3=\lambda_3y_3$,
	where $\lambda_1\lambda_2\lambda_3\neq0$
	and $\lambda_1+\lambda_2\neq0$.
	Indeed, if $\lambda=\mu $, take $y_2\in\mathcal X_1$ such that
	$y_1=(B-\lambda I)y_2\neq0$.
	Since $(B-\lambda I)^2=0$ and
	$\dim\mathcal X_1\geq3$, we have
	$\dim\ker(B-\lambda I)\geq2$.
	So  choose
	$y_3\in\ker(B-\lambda I)\setminus\operatorname{span}\{y_1\}$.
	If $\lambda\not=\mu $, then $B$ is diagonalizable. Take three
	linearly independent vectors
	$y_1,y_2,y_3\in\mathcal X_1$, reordering them if necessary
	so that $\lambda_1+\lambda_2\neq0$.
	Now, let $\mathcal M=\operatorname{span}\{z,y_1,y_2,y_3\}$
	and take  a closed complement $\mathcal L$ of
	$\mathcal M$ in $\mathcal X$.
	Consequently, define $T\in\mathcal B(\mathcal X)$ by
	$Ty_1=0$, $Ty_2=y_1$, $Tz=y_2$,
	$Ty_3=z$, and $T|_{\mathcal L}=0$.
	Then $\operatorname{rank}(T)=3$,
	$(AT+TA)y_1=0$,
	$(AT+TA)y_2=(\lambda_1+\lambda_2)y_1$,
	$(AT+TA)z=\lambda_2y_2+cy_1$
	and $(AT+TA)y_3=\lambda_3z$.
	Thus
	$$(AT+TA)^3y_3=\lambda_2\lambda_3
	(\lambda_1+\lambda_2)y_1\neq0\ \ {\rm and}\ \ (AT+TA)^4y_3=0.$$
If $\mathbb F=\mathbb R$ and $q(t)=(t-a)^2+b^2$ for some
	$a,b\in\mathbb R$ with $b\neq0$,
	set $J=\frac{1}{b}(B-aI)$.
	Then $J^2=-I$ and $B=aI+bJ$.
	Since $\dim\mathcal X_1\geq3$ and
	$J^2=-I$, we have $\dim\mathcal X_1\geq4$.
	Choose $y,w\in\mathcal X_1$ such that
	$y,Jy,w,Jw$  linearly independent.
	Put $\mathcal Y=\operatorname{span}\{y,Jy\}$ and
	$\mathcal W=\operatorname{span}\{w,Jw\}$.
	Since $\{By,BJy\}$ and $\{w,Bw\}$
	are bases of $\mathcal Y$ and $\mathcal W$,
	respectively, the vectors
	$z,By,BJy,w,Bw$ are linearly independent.
	Put $\mathcal M=\operatorname{span}\{z,By,BJy,w,Bw\}$
	and choose a closed complement
	$\mathcal L$ of $\mathcal M$ in $\mathcal X$.
	Define $T\in\mathcal B(\mathcal X)$ by
	$Tz=\frac{by+aJy}{a^2+b^2}$,
	$T(By)=z$, $T(BJy)=0$,
	$Tw=0$, $T(Bw)=y$,
	and $T|_{\mathcal L}=0$.
	Then $\operatorname{rank}(T)\leq3$.
	Since $B(by+aJy)=(a^2+b^2)Jy$,
one has  $(AT+TA)w=y$, $(AT+TA)y=z$,
	$(AT+TA)z=Jy$, and $(AT+TA)(Jy)=0$.
	Consequently,
	$$(AT+TA)^3w=Jy\neq0\ \ {\rm and}\ \ (AT+TA)^4w=0.$$
This completes the proof of Case 2.3.

 {\bf Case 2.4.} $A$ is algebraic of degree $4$.
	
	Choose $x\in\mathcal X$ such that
	$x,Ax,A^2x,A^3x$ are linearly independent,
	and write
	$A^4x=a_0x+a_1Ax+a_2A^2x+a_3A^3x$.
Let $\mathcal M=\operatorname{span}\{x,Ax,A^2x,A^3x\}$
	and choose a closed complement $\mathcal L$ of
	$\mathcal M$ in $\mathcal X$.
	Define $T\in\mathcal B(\mathcal X)$ by
	$Tx=x$, $T(A^3x)=x+Ax$,
	$T(Ax)=T(A^2x)=0$,
	and $T|_{\mathcal L}=0$.
	Then $\operatorname{rank}(T)=2$, and
	$$(AT+TA)^3(A^3x)=Ax\neq0,\ \ {\rm }\ \ (AT+TA)^4(A^3x)=0.$$
	
Combining the above all cases,
we achieve a contradiction under the assumption $\operatorname{rank}(A)\geq 2$. Therefore, $\operatorname{rank}(A)=1$,} finishing the proof of (3)$\Rightarrow$(1).
\end{proof}

 The following proposition characterizes the adjacency relation $\sim$ between rank-one nilpotent  operators by the ascent of Jordan product.

\begin{prop}\label{sim}
 {Assume that $A_1,A_2\in\mathcal N_1(\mathcal X)$ are linearly independent. Then the following assertions are equivalent.
	
	{\rm (1)} $A_1\sim A_2$.
	
	{\rm (2)} There exist $B_1,B_2\in\mathcal N_1(\mathcal X)$ such that any two of the operators $A_1,A_2,B_1,B_2$ are linearly independent, and for every $T\in\mathcal B(\mathcal X)$, if
	$\operatorname{asc}(TA_i+A_iT)\in\{2,3\}$ for  $i=1,2$,
	then
	$\operatorname{asc}(TB_1+B_1T)\in\{2,3\}$ or
	$\operatorname{asc}(TB_2+B_2T)\in\{2,3\}$.

	{\rm (3)} There exist $B_1,B_2\in\mathcal N_1(\mathcal X)$ such that any two of the operators $A_1,A_2,B_1,B_2$ are linearly independent, and for every $T\in\mathcal F_3(\mathcal X)$, if
	$\operatorname{asc}(TA_i+A_iT)\in\{2,3\}$ for  $i=1,2$,
	then
	$\operatorname{asc}(TB_1+B_1T)\in\{2,3\}$ or
	$\operatorname{asc}(TB_2+B_2T)\in\{2,3\}$.}
\end{prop}

\begin{proof}
 	$(1)\Rightarrow(2)$: Suppose first that $A_1$ and $A_2$ have the same range. Write $A_1=x\otimes f_1$ and $A_2=x\otimes f_2$, where $0\neq x\in\mathcal X$, $f_1,f_2\in\mathcal X^*$ are linearly independent, and $f_1(x)=f_2(x)=0$. Set $B_1=x\otimes(f_1+f_2)$ and $B_2=x\otimes(f_1+2f_2)$. Then $B_1,B_2\in\mathcal N_1(\mathcal X)$, and any two of $A_1,A_2,B_1,B_2$ are linearly independent.
	
	Let $T\in\mathcal B(\mathcal X)$ satisfy $\operatorname{asc}(TA_i+A_iT)\in\{2,3\}$ for  $i=1,2$. Put $L_i=TA_i+A_iT$, $i=1,2$. By Proposition \ref{main2}, one gets $L_1\neq0$, $L_2\neq0$, and $f_1(Tx)=f_2(Tx)=0$. Moreover, $TB_1+B_1T=L_1+L_2$ and $TB_2+B_2T=L_1+2L_2$. Note that these two operators cannot both be zero; otherwise $L_1+L_2=0$ and $L_1+2L_2=0$ would imply $L_2=0$, a contradiction. In addition, as $(f_1+f_2)(Tx)=0$ and $(f_1+2f_2)(Tx)=0$, by Proposition \ref{main2}, we obtain that $\operatorname{asc}(TB_1+B_1T)\in\{2,3\}$  if $L_1+L_2\not=0$, or $\operatorname{asc}(TB_2+B_2T)\in\{2,3\}$  if $L_1+2L_2\not=0$.
	
	Now suppose that $A_1$ and $A_2$ have the same kernel. Write $A_1=x_1\otimes f$ and $A_2=x_2\otimes f$, where $0\neq f\in\mathcal X^*$,  $x_1,x_2\in\mathcal X$ are linearly independent and  $f(x_1)=f(x_2)=0$. Set $B_1=(x_1+x_2)\otimes f$ and $B_2=(x_1+2x_2)\otimes f$. Again, $B_1,B_2\in\mathcal N_1(\mathcal X)$, and any two of $A_1,A_2,B_1,B_2$ are linearly independent.
	
	Let $T\in\mathcal B(\mathcal X)$ satisfy $\operatorname{asc}(TA_i+A_iT)\in\{2,3\}$ for $i=1,2$, and put $L_i=TA_i+A_iT$, $i=1,2$. By Proposition \ref{main2}, $L_i\neq0$ and $f(Tx_i)=0$, $i=1,2$. As  $TB_1+B_1T=L_1+L_2$, $TB_2+B_2T=L_1+2L_2$, $f(T(x_1+x_2))=0$ and $f(T(x_1+2x_2))=0$, it follows   from Proposition \ref{main2} again that $\operatorname{asc}(TB_1+B_1T)\in\{2,3\}$  if $L_1+L_2\not=0$, or $\operatorname{asc}(TB_2+B_2T)\in\{2,3\}$  if $L_1+2L_2\not=0$.

	$(2)\Rightarrow(3)$ is obvious.
	
	$(3)\Rightarrow(1)$: Assume, on the contrary, that $A_1\not\sim A_2$. Then we may write $A_1=x\otimes f$ and $A_2=y\otimes g$, where both $\{x,y\}$ and $\{f,g\}$ are linearly independent, and $f(x)=g(y)=0$. Take  $B_j=z_j\otimes h_j\in\mathcal N_1(\mathcal X)$ for  $j=1,2$  such that any two of $A_1,A_2,B_1,B_2$ are linearly independent.
	
	Choose $u_1,u_2\in\mathcal X$ such that $f(u_1)=g(u_2)=1$ and $f(u_2)=g(u_1)=0$. If $h_j\notin\operatorname{span}\{f,g\}$ for some $j\in\{1,2\}$,  take $u_3\in\ker f\cap\ker g$ such that $h_j(u_3)\ne0$; otherwise put $u_3=0$. Set $\mathcal U=\operatorname{span}\{u_1,u_2,u_3\}$. Then $\dim \mathcal U\le3$ and $\mathcal U^\perp\cap\operatorname{span}\{f,g,h_j\}=\{0\}$ for $j=1,2$.
	
	Define
	$$\mathcal M=\{T\in\mathcal B(\mathcal X):\operatorname{ran}(T)\subseteq \mathcal U,\ f(Tx)= g(Ty)=0\}.$$For $j=1,2$, set
	$$\mathcal H_j=\{T\in \mathcal M:h_j(Tz_j)=0\},\ \ \mathcal K_i=\{T\in \mathcal M:TA_i+A_iT=0\}.$$ Clearly, $\mathcal H_j$ and $\mathcal K_i$ are linear subspaces of $\mathcal M$, and every $T\in \mathcal M$ has rank at most $3$.
	
{\bf Claim 1.}  $\mathcal K_1$ and $\mathcal K_2$ are proper subspaces of $\mathcal M$.

Choose $\varphi\in\mathcal X^*$ such that $\varphi(x)=1$ and $\varphi(y)=0$, and let $T=u_2\otimes\varphi$. Then $f(Tx)=g(Ty)=0$,  and so $T\in \mathcal M$. Moreover, $TA_1+A_1T=u_2\otimes f\ne0$. Hence $\mathcal K_1\ne \mathcal M$.

Similarly,  one can show that $\mathcal K_2\ne \mathcal M$.
	
{\bf Claim 2.} $\mathcal H_1$ and $\mathcal H_2$ are  proper subspaces of $\mathcal M$.
	
	Suppose otherwise that $\mathcal H_j=
	\mathcal M$ for some $j\in\{1,2\}$.
	Let
	$$\mathcal L=\{T\in\mathcal B(\mathcal X):
	\operatorname{ran}T\subseteq \mathcal U\}.$$
	Define linear functionals $\phi_1$, $\phi_2$ and $\psi_j$
	on $\mathcal L$ respectively by $$\phi_1(T)=f(Tx),\ \  \phi_2(T)=g(Ty),\ \ \psi_j(T)=h_j(Tz_j).$$ Since
	$\mathcal M=\ker\phi_1\cap\ker\phi_2\subseteq\ker\psi_j$,
	there exist some $\alpha_j,\beta_j\in\mathbb F$ such that
	$\psi_j=\alpha_j\phi_1+\beta_j\phi_2$.
	Take any $u\in \mathcal U$ and
	$\varphi\in\mathcal X^*$. For $T=u\otimes\varphi$, we obtain
	$h_j(u)\varphi(z_j)=\alpha_jf(u)\varphi(x)
	+\beta_jg(u)\varphi(y)$, which implies
	$$h_j(u)z_j=\alpha_jf(u)x+\beta_jg(u)y,\ \  \forall u\in \mathcal U.$$
	
	Set $S=z_j\otimes h_j-\alpha_jx\otimes f
	-\beta_jy\otimes g=B_j-\alpha_jA_1-\beta_j A_2 \in\mathcal F_3(\mathcal X)$.
	Then $S|_\mathcal U=0$ and
	$\operatorname{ran}S^*\subseteq
	\operatorname{span}\{f,g,h_j\}$.
	Consequently,
	$\operatorname{ran}S^*\subseteq
	\mathcal U^\perp\cap\operatorname{span}\{f,g,h_j\}
	=\{0\}$, which implies $S=0$.
	That is, $B_j
	=\alpha_jA_1+\beta_jA_2$.
It follows that $B_j$ is a scalar multiple of either $A_1$ or $A_2$, contradicting the assumption in (3). Hence $\mathcal H_j\neq\mathcal M$ for $j=1,2$.

Note that a vector space over $\mathbb F$ cannot be the union of finitely many proper linear subspaces. Hence $\mathcal M\ne \mathcal H_1\cup \mathcal H_2\cup \mathcal K_1\cup \mathcal K_2$. Choose $T\in \mathcal M\setminus(\mathcal H_1\cup \mathcal H_2\cup \mathcal K_1\cup \mathcal K_2)$. Then $\operatorname{rank}(T)\le3$, $f(Tx)=g(Ty)=0$ and $TA_i+A_iT\ne0$, $i=1,2$. It follows from Proposition \ref{main2} that $\operatorname{asc}(TA_i+A_iT)\in\{2,3\}$, $i=1,2$. However,  as $h_j(Tz_j)\ne0$ for $j=1,2$, by Proposition \ref{main2} again, $\operatorname{asc}(TB_j+B_jT)\notin\{2,3\}$, $j=1,2$, which contradicts (3). Therefore $A_1\sim A_2$.
	
	The proof is completed.

\end{proof}

 Before the end of this section, we characterize the linear dependency of  operators by the ascent of Jordan products.

\begin{prop}\label{linear combination}
	Assume that  $A, B \in \mathcal B(\mathcal X)\setminus\mathbb F I$. The following statements are equivalent.

{\rm (1)} $B=\lambda A+\mu I$ for some scalars $\lambda,\mu\in\mathbb F$.

{\rm (2)} For any $N \in  \mathcal N_1(\mathcal X)$,
$\operatorname{asc}(AN + NA) \in \{2,3\} $ if and only if $\operatorname{asc}(BN + NB)  \in \{2,3\}$.
\end{prop}

\begin{cor}\label{linearly dependent}
	Assume that $A, B \in \mathcal N_1 (\mathcal X)$. Then the following assertions are equivalent.
	
	{\rm (1)} $A= \lambda B $ for some  nonzero  $\lambda \in\mathbb F$.

{\rm (2)} $\operatorname{asc}(AN+ NA)=\operatorname{asc}(BN + NB)$ holds for all $N \in  \mathcal N_1(\mathcal X)$.
	
	{\rm (3)} For any  $N \in  \mathcal N_1(\mathcal X)$,   $\operatorname{asc}(AN+ NA)=2 $ if and only if  $\operatorname{asc}(BN + NB)=2$.
\end{cor}

\begin{proof}
	(3) $\Rightarrow$ (1):  By Proposition \ref{two1}, $\operatorname{asc}(AN+ NA)\in\{1,2\}$. So $\operatorname{asc}(AN+ NA)\in\{2,3\}$ if and only if $\operatorname{asc}(AN+ NA)=2$. Thus, the hypothesis of Proposition 4.7 is satisfied,  and consequently,  $B=\alpha A+\beta I$ for some $\alpha,\beta\in\mathbb F$. As $B$ is nilpotent, we must have $\beta=0$.
	
	(1) $\Rightarrow$ (2): Assume that  $A= \lambda B $ for some  nonzero  $\lambda \in\mathbb F$. By Proposition \ref{p1}(1), we have
	$$\begin{aligned}
		\operatorname{asc}(AN + NA)
		= \operatorname{asc}(\lambda BN  + \lambda NB)
		= \operatorname{asc}(BN+NB)
	\end{aligned}$$
for every $N \in  \mathcal N_1(\mathcal X)$.

(2) $\Rightarrow$ (3) is obvious.
\end{proof}

\begin{proof}[Proof of Proposition \ref{linear combination}]
	To prove the proposition, we only need to show that  $ Bx \in \operatorname{span}\{x, Ax\} $ holds for all $x \in \mathcal X$. Thus, by \cite[Lemma 3.3]{XF2012}, it holds that $B=\lambda A+\mu I$ for some scalars $\lambda,\mu\in\mathbb F$.
	
	In the sequel,  assume that there exists some $x_0\in\mathcal X$ such that $ Bx_0 \notin \operatorname{span}\{x_0, Ax_0\}$. We will obtain a contradiction by  considering two cases.
	
	{\bf Case 1.} $x_0$ and $ Ax_0$ are linearly independent.
	
	In this case,  $ x_0$, $Ax_0$ and $Bx_0$ are linearly independent. Choose $f_0 \in \mathcal X^* $ such that $f_0(x_0)=0$, $f_0(Ax_0)= 0$ and  $f_0(Bx_0) \neq 0$. It is obvious that $Ax_0 \otimes f_0+x_0 \otimes A^*f_0 \neq 0$. Let  $N=x_0 \otimes f_0\in\mathcal N_1(\mathcal X)$. Then  $\operatorname{asc}(AN + NA)\in \{2,3\}$ but  $\operatorname{asc}(BN + NB) \notin \{2,3\}$ by Proposition \ref{main2}, a contradiction.
	
	{\bf Case 2.} $x_0$ and $ Ax_0$ are linearly dependent.
	
	Write $ Ax_0 = \gamma x_0 $ for some $\gamma\in\mathbb F$. Since $x_0$ and $Bx_0$ are linearly independent, we may choose $ f_0,  g_0 \in \mathcal X^*$ such that $f_0(x_0) = g_0(Bx_0) = 1$ and $f_0(Bx_0) = g_0(x_0) = 0$. Under the space decomposition $ \mathcal X = \operatorname{span}\{x_0\}\oplus \operatorname{span}\{Bx_0\} \oplus (\operatorname{ker}(f_0)\cap \operatorname{ker}(g_0)) $, $A$  has the   matrix representation
	$ A = \begin{pmatrix}
		\gamma & a_{12} & A_{13} \\
		0 & a_{22} & A_{23} \\
		0 & A_{32} & A_{33}
	\end{pmatrix}.$
	Thus $ \mathcal X^* = \operatorname{span}\{f_0\}\oplus \operatorname{span}\{g_0\} \oplus \operatorname{span}\{x_0,Bx_0\}^{\perp} $ and
	$A^* = \begin{pmatrix}
		\gamma & 0 & 0 \\
		a_{12} & a_{22} & A_{32}^* \\
		A_{13}^* & A_{23}^* & A_{33}^*
	\end{pmatrix}.$
	
	If $ A^* g_0 \neq -\gamma g_0 $, then by taking $ N = x_0 \otimes g_0\in\mathcal N_1(\mathcal X)$, we get  $AN + NA \neq 0 $. By  Proposition \ref{main2},  $\operatorname{asc}(AN + NA)\in \{2,3\}$ and  $\operatorname{asc}(BN + NB) \notin \{2,3\}$, a contradiction.
	
	If  $A^* g_0 = -\gamma g_0 $ and $A_{32}^* \neq 0 $, then there exists some $ h_0 \in \operatorname{span}\{x_0,Bx_0\}^{\perp} $ such that $ A_{32}^* h_0 \neq  0$. So
	$A^*h\ne-\gamma h$ with $h=(0\oplus 0\oplus h_0)\in \mathcal X^*$.  Note that $A^*(g_0+h) \neq -\gamma (g_0+h) $, $(g_0+h)(x_0)=(g_0+h)(Ax_0)=0$ and $(g_0+h)(Bx_0)=1$. Now, take $ N = x_0 \otimes (g_0+h)\in\mathcal N_1(\mathcal X)$, and then $AN + NA \neq 0 $. By Proposition  \ref{main2} again,   $\operatorname{asc}(AN + NA)\in \{2,3\}$ and  $\operatorname{asc}(BN + NB) \notin \{2,3\}$, a contradiction.
	
	If $A^* g_0 = -\gamma g_0 $ and $A_{32}^* = 0 $,  we consider  $ A_{33}^* $.
	Let $I' $ be the identity operator on $\operatorname{span}\{x_0,Bx_0\}^{\perp}$. If $ A_{33}^* \neq -\gamma I' $, there exists some $ h_1 \in \operatorname{span}\{x_0,Bx_0\}^{\perp} $ such that $A_{33}^* h_1 \neq -\gamma h_1 $. Thus $A^* h \neq -\gamma h$ with $h=0\oplus 0\oplus h_1\in\mathcal X^*$. Letting $ N = x_0 \otimes (g_0+h)\in\mathcal N_1(\mathcal X)$, by Proposition \ref{main2}, we have  $\operatorname{asc}(AN + NA)\in \{2,3\}$ and  $\operatorname{asc}(BN + NB) \notin \{2,3\}$, a contradiction.
	If $ A_{33}^* = -\gamma I'$,
	we can choose $ h_2 \in \mathcal{X^*} $ such that $ h_2(x_0) = h_2(Bx_0) = 0$ as the linearly independence of  $ x_0$ and $Bx_0 $. Then $ h_2 \in \operatorname{span}\{x_0,Bx_0\}^{\perp}$. It is easily checked that ${A^*}{h_2} =- \gamma{h_2}$. Now, letting $N = x_0 \otimes h_2\in\mathcal N_1(\mathcal X) $, by Proposition \ref{main2},  we still get $\operatorname{asc}(BN + NB) \in \{2,3\}$ and  $\operatorname{asc}(AN + NA)\notin \{2,3\}$,  a contradiction again.
	
	The proof is completed.
\end{proof}

\begin{rem}\label{rem4.9}
 All the preceding results concerning the ascent in Sections 3-4 remain valid for the descent, except that ``injective'' in Proposition 4.2(4) should be replaced by ``surjective''. Indeed, this is due to the facts that  $\operatorname{asc}(S)=\operatorname{desc}(S)$ holds for every finite-rank operator $S$, and $AT+TA$ has finite rank whenever either $A$ or $T$ has finite rank. For instance, the analogous of Proposition \ref{rank1} is
\end{rem}

\begin{prop}\label{descrank1}
	Assume that $0\neq A\in\mathcal B(\mathcal X)$. Then the following statements are equivalent.
	
	{\rm (1)} $\operatorname{rank}(A)=1$.
	
	{\rm (2)} $\operatorname{desc}(AT+TA)\in\{1,2,3\}$
	for every $T\in\mathcal B(\mathcal X)$.
	
	{\rm (3)} $\operatorname{desc}(AT+TA)\in\{1,2,3\}$
	for every $T\in\mathcal F_3(\mathcal X)$.
\end{prop}

\section{An  application:  Characterizing maps preserving  the ascent (descent) of Jordan product of operators}

In this section,  we give an application of the results in Sections 2-4 to characterize all  maps on $\mathcal B(\mathcal X)$ that preserve the ascent (descent) of Jordan product on $\mathcal B(\mathcal X)$, which give a characterization of ring isomorphisms of $\mathcal B(\mathcal X)$ when $\dim \mathcal X=\infty$ and  Jordan ring isomorphisms of $\mathcal B(\mathcal X)$ when $\dim\mathcal X<\infty$ up to a functional   that is non-zero at every operator.

The following is  our main result in this section.

\begin{thm}\label{1}
	Let $\mathcal X$ be a real or complex Banach space with $\dim \mathcal X\geq 3$. Assume that $\phi : \mathcal B(\mathcal X) \rightarrow \mathcal B(\mathcal X)$ is a  map with range containing $\mathcal F_3(\mathcal X)$.
	Then $\phi$ satisfies
	$$ \operatorname{asc}(AB + BA) = \operatorname{asc}(\phi(A) \phi(B) + \phi(B) \phi(A))\quad {\rm for\ all}\ \ A, B \in {\mathcal B}(\mathcal{X}),$$
	if and only if one of the following statements holds:
	
	{\rm (1)} if $ \mathcal X$ is infinite-dimensional, then there exist an invertible bounded linear or conjugate-linear operator $T : \mathcal X \to \mathcal X$ and a   functional  $f : \mathcal B(\mathcal X) \rightarrow\mathbb{F}^* $ such that
	$$\phi(A) = f(A) TAT^{-1}\ \ {\rm for\ all} \ \ A \in \mathcal B(\mathcal X);$$

	{\rm (2)}  if $ \mathcal X$ is finite-dimensional, regarding $\mathcal B(\mathcal X)$ as $\mathcal M_n(\mathbb F)$,  then there exist  an invertible matrix $T\in \mathcal M_n(\mathbb F)$, an automorphism $ \tau:\mathbb{F}\rightarrow\mathbb F $ and a   functional  $f: \mathcal M_n(\mathbb F) \to \mathbb{F}^* $ such that either
	$$\phi(A)= f(A)T[\tau(a_{ij})]T^{-1}\ \ {\rm for\ all}\ \  A=[a_{ij}] \in \mathcal M_n(\mathbb F), $$
	or
	$$\phi(A)= f(A)T[\tau(a_{ij})]^{\operatorname{tr}} T^{-1}\ \ {\rm for\ all}\ \   A=[a_{ij}] \in \mathcal M_n(\mathbb F),$$
	where $A^{\operatorname{tr}}$ denotes the transpose of $A$.
\end{thm}

To prove the theorem, the following lemma is needed.

\begin{lem}\label{C}
	{Let $\mathcal X$ be an infinite-dimensional Banach space. Then there exists some
		$C\in\mathcal B(\mathcal X)$ such that $\operatorname{asc}(C^*)\neq\operatorname{asc}(C)$.}
\end{lem}

\begin{proof}
	{ By Mazur's theorem, every infinite-dimensional Banach space $\mathcal X$ contains an infinite-dimensional closed subspace with a normalized Schauder basis $\{x_n\}_{n\geq1}$. By  Hahn-Banach theorem, there exist functionals $\{f_n\}_{n\geq1}\subseteq\mathcal X^*$ such that $f_n(x_m)=\delta_{nm}$ for all $m,n\geq1$ and $\sup\limits_{n\geq1}\|f_n\|<\infty$. Choose a sequence of nonzero scalars $\{a_n\}_{n\geq1}$ satisfying $\sum\limits_{n=1}^\infty |a_n|\|f_n\|<\infty$, and define $C\in\mathcal B(\mathcal X)$ by
	$Cx=\sum\limits_{n=1}^\infty a_nf_n(x)x_{n+1}$, $x\in\mathcal X$.
		For every $x\in\mathcal X$,
		$C^2x=\sum\limits_{n=1}^{\infty}a_na_{n+1}f_n(x)x_{n+2}.$
		Since $\{x_n\}$ is a basic sequence and $a_n\neq0$ for all $n$, we have $C^2x=0$ if and only if
		$f_n(x)=0$ for all $n$. But $f_n(x)=0$ for all $n$ is equivalent
		to $Cx=0$.
		Hence $\ker C=\ker C^2$, and so
		$\operatorname{asc}(C)\leq1$.
		On the other hand, for every $n\geq1$ and $x\in\mathcal X$,  one has
		$(C^*f_n)(x)=f_n(Cx)$.
		Thus
		$$(C^*f_{n+1})(x)=\sum_{k=1}^{\infty}a_kf_k(x)f_{n+1}(x_{k+1})
		=a_nf_n(x),\quad n\geq1,$$
		and
		$$(C^*f_1)(x)=\sum_{k=1}^{\infty}a_kf_k(x)f_1(x_{k+1})=0.$$
		Hence $C^*f_1=0$ and $C^*f_{n+1}=a_nf_n$ for $n\geq1$.
		Consequently,
		$(C^*)^{n+1}f_{n+1}=0$
		while
		$(C^*)^nf_{n+1}=a_na_{n-1}\cdots a_1f_1\neq0.$
		Thus $\ker (C^*)^n\subsetneq\ker (C^*)^{n+1}$ for $n\geq1$, which yields
		$\operatorname{asc}(C^*)=\infty$.
		Therefore $\operatorname{asc}(C^*)\neq\operatorname{asc}(C)$.}
\end{proof}

\begin{proof}[Proof of Theorem \ref{1}]
	
	By Proposition \ref{p1}(1)-(2),  and  by the multiplicativity  and $\tau$-linearity for (2), it is easy to check that $\phi$
	preserves the ascent of Jordan product if $\Phi$ has the form in (1) or (2). Hence the ``if'' part is true.
	
	In the sequel we  prove the ``only if'' part  by several claims.
	
	Assume that $\phi : \mathcal B(\mathcal X) \rightarrow \mathcal B(\mathcal X)$  is a  map with range containing $\mathcal F_3(\mathcal X)$ satisfying
	\begin{equation} \label{3.1}
		\operatorname{asc}(AB + BA) = \operatorname{asc}(\phi(A) \phi(B) + \phi(B) \phi(A))\quad {\rm for\ all}\ \ A, B \in \mathcal B(\mathcal X).
	\end{equation}

{ {\bf Claim 1.}   $\mathbb F^*I\subseteq{\rm ran}\ \phi$, and moreover, $\phi(\mathbb F^*I)=\mathbb F^*I$.
	
Take any  $\lambda\in\mathbb F^*$ and $T\in\mathcal F_3(\mathcal X)$. Since $\operatorname{ran}\phi\supseteq\mathcal F_3(\mathcal X)$, there exists some  $S\in\mathcal B(\mathcal X)$ such that $\phi(S)=T$. By Eq.\eqref{3.1}, we have
	$$
	\begin{aligned}
		\operatorname{asc}(\phi(\lambda I)T+T\phi(\lambda I))
		=\operatorname{asc}(\lambda IS+S\lambda I)
		=\operatorname{asc}(S).
	\end{aligned}
	$$
	Applying Proposition $\ref{ascent power}$,  one obtains
	$$
	\begin{aligned}
		\operatorname{asc}\bigl((\phi(\lambda I)T+T\phi(\lambda I))^2\bigr)
		&=\left\lceil
		\frac{\operatorname{asc}(\phi(\lambda I)T+T\phi(\lambda I))}{2}
		\right\rceil\\
		&=\left\lceil\frac{\operatorname{asc}(S)}{2}\right\rceil
		=\operatorname{asc}(S^2).
	\end{aligned}
	$$
As $\operatorname{asc}(S^2)=\operatorname{asc}(\phi(S)^2)$ and
	$\operatorname{asc}(\phi(S)^2)=\operatorname{asc}(T^2)$, one  achives
	$$\operatorname{asc}\bigl((\phi(\lambda I)T+T\phi(\lambda I))^2\bigr)=\operatorname{asc}(T^2)\ \ {\rm for\ all}\ \ T\in\mathcal F_3(\mathcal X).$$ It follows from  Proposition $\ref{I}$ that     $\phi(\lambda I)\in\mathbb F^*I$.
	
	Conversely, suppose that $\phi(A)\in\mathbb F^*I$ for some $A\in\mathcal B(\mathcal X)$. Since $\phi(I)\in\mathbb F^*I$, Eq.\eqref{3.1} implies that
	$$
	\operatorname{asc}(B)
	=\operatorname{asc}(IB+BI)
	=\operatorname{asc}(\phi(I)\phi(B)+\phi(B)\phi(I))
	=\operatorname{asc}(\phi(B))
	$$
holds 	for every $B\in\mathcal B(\mathcal X)$. Consequently,
	$$
	\begin{aligned}
		\operatorname{asc}(AB+BA)
		&=\operatorname{asc}(\phi(A)\phi(B)+\phi(B)\phi(A))\\
		&=\operatorname{asc}(\phi(B))
		=\operatorname{asc}(B)
	\end{aligned}
	$$
holds	for all $B\in\mathcal B(\mathcal X)$. By Proposition $\ref{I}$ again, we obtain $A\in\mathbb F^*I$.
}

		{\bf Claim 2.} $A=0$ if and only if $\phi(A)=0$.
		
		For any $T\in\mathcal F_3(\mathcal X)$, since
		$\mathcal F_3(\mathcal X)
		\subseteq\operatorname{ran}\phi$, there exists
		$B\in\mathcal B(\mathcal X)$ such that $\phi(B)=T$. By
		Eq.~\eqref{3.1},
		$\operatorname{asc}(\phi(0)T+T\phi(0))
		=\operatorname{asc}(0B+B0)=1$.
		Hence Proposition~\ref{0} gives $\phi(0)=0$.
		Thus $A=0$ implies $\phi(A)=0$. Conversely, if $\phi(A)=0$, then
		for every $T\in\mathcal F_3(\mathcal X)$, Eq.~\eqref{3.1} yields
		$\operatorname{asc}(AT+TA)
		=\operatorname{asc}(\phi(A)\phi(T)+\phi(T)\phi(A))
		=1$.
		By Proposition~\ref{0}, $A=0$. Therefore
		$A=0$ if and only if $\phi(A)=0$.

	{\bf Claim 3.} $\phi$ preserves rank-one operators in both directions, and moreover, preserves rank-one nilpotent operators in both directions.
	
	By Eq.~\eqref{3.1}, Proposition \ref{rank1}, and Claim 2, for any nonzero
	$A\in\mathcal B(\mathcal X)$,
	$A\in\mathcal F_1(\mathcal X)$ if and only if
	$$
	\operatorname{asc}(\phi(A)\phi(B)+\phi(B)\phi(A))
	=
	\operatorname{asc}(AB+BA)\in\{1,2,3\}
	$$
	for all $B\in\mathcal B(\mathcal X)$. Since $A\neq0$, Claim 2 implies
	$\phi(A)\neq0$. Thus, by the assumption that the range of $\phi$ contains all operators of rank not greater than $3$ and applying Proposition \ref{rank1} again, the above condition holds if and only if $\phi(A)$ is of rank one.
	
	In addition, by Claim 1, $\phi(I)\in\mathbb F^*I$, and hence
	$$
	\operatorname{asc}(\phi(C))
	=
	\operatorname{asc}(\phi(C)\phi(I)+\phi(I)\phi(C))
	=
	\operatorname{asc}(2C)
	=
	\operatorname{asc}(C)
	\ \ {\rm for\ all}\ C\in\mathcal B(\mathcal X).
	$$
	Thus, as $\phi$ preserves rank-one operators in both directions, by Proposition \ref{p1}(7), $\phi$ preserves also the rank-one nilpotent operators in both directions.

	{\bf Claim 4.}  For any  rank-one nilpotent operators $N_1,N_2 \in \mathcal N_1(\mathcal X)$,  $N_1$ and $N_2$ are linearly dependent if and only if $\phi(N_1)$ and $\phi(N_2)$ are linearly dependent.
	
	For any  $N_1, N_2 \in \mathcal N_1(\mathcal X)$.  If $N_1 = \alpha N_2$ for some  $\alpha \in \mathbb F^*$, by Corollary \ref{linearly dependent},   one gets
	$$\operatorname{asc}(N_1N+NN_1)=2 \Leftrightarrow \operatorname{asc}(N_2N+NN_2)=2\ \ {\rm for\ all}\ N\in\mathcal N_1(\mathcal X). $$
	By  Claim 3 and Eq.\eqref{3.1}, the above equation implies
	$$\operatorname{asc}(\phi(N_1)N+N\phi(N_1)) =2 \Leftrightarrow \operatorname{asc}(\phi(N_2)N+N\phi(N_2))=2\ \ {\rm for\ all}\ N\in\mathcal N_1(\mathcal X).$$
	Applying Claim 3 and Corollary \ref{linearly dependent} again, one obtains  $\phi(N_1) = \beta \phi(N_2)$ for some $\beta \in \mathbb F^*$.
	
	On the other hand, if  $\phi(N_1)$ and $\phi(N_2)$ are linearly dependent, by a similar argument to that of the above, one can show that $N_1$ and $N_2$ are also linearly dependent. So the claim is true.

	{\bf Claim 5.}  For any rank-one nilpotent operators $N_1,N_2 \in \mathcal N_1(\mathcal X)$, $N_1\sim N_2$  if and only if $\phi(N_1)\sim \phi(N_2)$.
	
	By  Claims 3-4  and Proposition \ref{sim}, this claim is obvious.
	
	Now, combining   Claims 1-5 and  \cite[Proposition 2.13]{RH2024}, the following claim holds.
	
	{\bf Claim 6.} 	If $\dim\mathcal X=\infty$, then either there exist an invertible bounded linear or conjugate-linear operator $T: \mathcal{X} \to \mathcal X$ and a map $f: \mathcal N_1(\mathcal X) \to \mathbb F^*$ such that
	\begin{equation}\label{4}	\phi(N) = f(N) TNT^{-1}\ \ {\rm for\ all}\ \ N \in \mathcal N_1(\mathcal X),
	\end{equation}
	or there exists an invertible bounded linear or conjugate-linear operator $T: \mathcal X^* \to \mathcal X$ and  a map $f: \mathcal N_1(\mathcal X) \to \mathbb F^*$ such that
	\begin{equation}\label{5}
		\phi(N) = f(N) TN^*T^{-1}\ \ {\rm for\ all}\ \ N \in \mathcal N_1(\mathcal X),
	\end{equation}
	and in the latter case, $\mathcal X$ is reflexive;
	
	if $\dim\mathcal X=n<\infty$, regarding $\mathcal B(\mathcal X)$ as $\mathcal M_n(\mathbb F)$,  then there exists  an invertible matrix $T\in \mathcal M_n(\mathbb F)$, an automorphism $ \tau:\mathbb{F}\rightarrow\mathbb F $ and a map $f: \mathcal{N}_1(\mathcal{X}) \to \mathbb{F}^* $ such that either
	\begin{equation}\label{6}
		\phi(A)= f(A)T[\tau(a_{ij})]T^{-1}\ \ {\rm for\ all}\ \  A=[a_{ij}] \in\mathcal N_1(\mathcal{X}),
	\end{equation}
	or
	\begin{equation}\label{7}
		\phi(A)= f(A)T[\tau(a_{ij})]^{\operatorname{tr}} T^{-1}\ \ {\rm for\ all}\ \   A=[a_{ij}] \in \mathcal N_1(\mathcal X).
	\end{equation}

	{\bf Claim 7.}  $\phi$ has the forms in Theorem \ref{1}, and therefore, the theorem holds.

	Assume that $\dim \mathcal X=\infty$. Then by Claim 6,  Eq.\eqref{4} or Eq.\eqref{5} holds.   Replacing $\phi$ with $f^{-1}(\cdot) T^{-1}\phi(\cdot) T$, which  obviously still preserves the ascent of Jordan products, we may assume that  either $ \phi(N)= N$ holds for all $N\in\mathcal N_1(\mathcal X)$ or $ \phi(N)= N^*$ holds for all $N\in\mathcal N_1(\mathcal X)$, and by Claims 1-2, $\phi (A)=A$
if $A\in\mathbb F I$.

If $ \phi(N)= N$ holds for all $N\in\mathcal N_1(\mathcal X)$, then, for any $A\in\mathcal B(\mathcal X)$,
$$\
	\operatorname{asc}(AN + NA)
	= \operatorname{asc}(\phi(A)N + N \phi(A))$$ 
holds for all $N\in\mathcal N_1(\mathcal X)$. It follows from Proposition \ref{linear combination} that $$\phi(A)=\beta_AA +\gamma_AI $$
holds for all $A\in\mathcal B(\mathcal X)$. By Claim 1,  $\beta_A\not=0$.
	Still, we may suppose $\beta_A=1$ for every $A$. Thus 
$$\phi(A)=A +\gamma_AI $$
holds for all $A\in\mathcal B(\mathcal X)$.

 	Note that,
	\begin{equation}\label{f}
		\operatorname{asc}(F)=\operatorname{asc}(F^*) \text{ holds for all finite-rank operators } F \in \mathcal{B}(\mathcal{X}).
	\end{equation}
A similar argument ensures that, if  
$ \phi(N)= N^*$ holds for all $N\in\mathcal N_1(\mathcal X)$, then 
$$\phi(A)=A^* +\gamma_AI $$
holds for all $A\in\mathcal B(\mathcal X)$.

Next, we prove that  $\gamma_A=0$ for every $A$. 

	For any $A\in\mathcal B(\mathcal X)$, denote by $A^\dagger=A$ or $A^*$. Then, we may write the two possible forms of $\phi$ above as 
\if false 
by Eq.\eqref{3.1}, we have $$\
	\operatorname{asc}(AN + NA)
	= \operatorname{asc}(\phi(A)N^\dagger + N^\dagger \phi(A))
	= \operatorname{asc}(\phi(A)^\dagger N + N \phi(A)^\dagger).$$
 It follows from Proposition \ref{linear combination} that $\phi(A)^\dagger=\beta_AA +\gamma_AI $, that is,\fi
	\begin{equation}\label{A}
		\phi(A)= A^\dagger +\gamma_A I \ \ \mbox{for all}\ \ A\in\mathcal B(\mathcal X),   \ {\rm where} \ \gamma_A \in \mathbb{F}.
	\end{equation}
	
  Since $\phi$ preserves rank-one operators in both directions by Claim 3,  one has
	$$\gamma_{x\otimes f}=0\ \ {\rm for\ all}\ x\otimes f\in\mathcal F_1(\mathcal X),$$
	and so
	\begin{equation}\label{x}
		\phi(x\otimes f)=(x\otimes f)^\dag\ \ {\rm for\ all}\ x\otimes f\in\mathcal F_1(\mathcal X).
	\end{equation}
	
For any $A\in\mathcal B(\mathcal X)\setminus \mathbb F I$ with rank$(A)\geq 2$,  we check $\gamma_A=0$  by two cases.

	{\bf Case 1.} There exists  some $x_0 \in \mathcal{X} $ such that $x_0$, $Ax_0$ and $A^2x_0$ are linearly independent.
	
	Since $x_0$, $Ax_0$ and $A^2x_0$ are linearly independent, we may take a functional $f_0 \in \mathcal{X}^*$ such that $f_0(x_0) \neq 0$, $f_0(Ax_0)=0$ and $f_0(A^2x_0) = 0$. By Eq.\eqref{3.1} and  Eqs.\eqref{f}-\eqref{x}, one has
	$$\begin{array}{rl} \operatorname{asc}(Ax_0 \otimes f_0 + x_0 \otimes A^*f_0)
		=& \operatorname{asc}(\phi(A)(x_0 \otimes f_0)^\dag+ (x_0 \otimes f_0)^\dag\phi(A))\\
		=&\operatorname{asc}((A +\gamma_A I))(x_0 \otimes f_0) + (x_0 \otimes f_0)((A +\gamma_A I))\end{array}$$
	
	If $A^*f_0=\lambda f_0 $ for some scalar $\lambda \in \mathbb{F}$,  then $0=f_0(Ax_0)=\lambda f_0(x_0)$. So
	$A^*f_0=0$. Thus,  the above equation implies  $$2=\operatorname{asc}(Ax_0 \otimes f_0)=\operatorname{asc}((Ax_0 +2\gamma_Ax_0 )\otimes f_0).$$
	It follows from  Proposition \ref{p1}(7) that $\gamma_A=0$.

	If $A^*f_0$ and $f_0 $ are linearly independent, then $\operatorname{rank}(Ax_0 \otimes f_0 + x_0 \otimes A^*f_0)=2$. By Proposition \ref{main1}, one gets
	$$3=\operatorname{asc}(Ax_0 \otimes f_0 + x_0 \otimes A^*f_0)=\operatorname{asc}((A +\gamma_A I))(x_0 \otimes f_0) + (x_0 \otimes f_0)((A +\gamma_A I)).$$
	Applying  Proposition \ref{main1} again, it follows that
	$f_0(A x_0+\gamma_A x_0)=0$, and so $\gamma_A=0$.
	
	{\bf Case 2.} For all $x \in \mathcal{X} $, $x$, $Ax$ and $A^2x$ are linearly dependent.
	
	In this case, by \cite[Kaplansky's Theorem]{HR1973},  $A$ is an algebraic operator.
		Take an operator degree  not 2  with rank $< \infty $, or   a rank-one operator  $T \in \mathcal B(\mathcal X)$. By Proposition \ref{rank1}  and  Case 1, we have $\phi(T)=T^\dag$. Thus, by Eq.\eqref{3.1} and Eqs.\eqref{f}-\eqref{A}, one achieves
	$$ \operatorname{asc}(AT + TA)= \operatorname{asc}(AT + TA+2\gamma_AT).$$
	It follows from Proposition \ref{scalar perturbation} that  $\gamma_A= 0$.
	
	Thus, we obtain that either there exist an invertible bounded linear or conjugate-linear operator $T : \mathcal X \to \mathcal X$ and a map $f : \mathcal B(\mathcal X) \rightarrow\mathbb{F}^* $ such that
	$$\phi(A) = f(A) TAT^{-1}\ \ {\rm for\ all} \ \ A \in \mathcal B(\mathcal X), $$
	or $\mathcal X$ is reflexive, and there exist an invertible bounded linear or conjugate-linear operator $T : \mathcal{X}^* \to \mathcal{X}$ and a map $f : \mathcal B(\mathcal X) \to \mathbb{F}^*$ such that
	$$\phi(A) = f(A)TA^*T^{-1}  \ \ {\rm for\ all}\ \   A \in \mathcal B(\mathcal X).$$
	Moreover, by Lemma \ref{C}, it is clear that the  last case  cannot occur.
	Thus  the statement (1) holds for the case $\dim \mathcal X=\infty$.
	
	Next, assume $3\leq \dim \mathcal X=n<\infty$. By Claim 6,t Eq.\eqref{6} or Eq.\eqref{7} holds. Still, we can assume that either $ \phi(N)= [\tau(n_{ij})]$ holds for all $N=[n_{ij}]\in\mathcal N_1(\mathcal X)$ or $ \phi(N)= [\tau(n_{ij})]^{\operatorname{tr}}$ holds for all $N=[n_{ij}]\in\mathcal N_1(\mathcal X)$. Note that $[\tau(n_{ij})] \in\mathcal N_1(\mathcal X) $ and  $\operatorname{asc}([a_{ij}])=\operatorname{asc}([\tau(a_{ij})]) $ holds for all $A=[a_{ij}] \in \mathcal M_n(\mathbb{F})$.
	So, for any $A=[a_{ij}]\in\mathcal M_n(\mathbb{F})$,  by Eq.\eqref{3.1}, one can obtain
	$$\operatorname{asc}([\tau(a_{ij})][\tau(n_{ij})]+[\tau(n_{ij})][\tau(a_{ij})])
	=  \operatorname{asc}(\phi(A)^\dagger [\tau(n_{ij})] + [\tau(n_{ij})] \phi(A)^\dagger),$$
	which and Proposition \ref{linear combination} imply  $\phi(A)^\dagger=\beta_A[\tau(a_{ij})] +\gamma_AI $, that is,
	\begin{equation*}\label{8}
		\phi(A)=\beta_A A^\dagger +\gamma_A I \ \ \mbox{for all}\ \ A\in\mathcal M_n(\mathbb{F}),   \ {\rm where} \ \beta_A, \gamma_A \in \mathbb{F}.
	\end{equation*}
	Now, by a similar argument to that of the infinite-dimensional case, one can show  $\gamma_A=0$. Hence the statement (2) holds.
	
	The proof of Theorem \ref{1} is completed.

\end{proof}

 An analogous result to Theorem \ref{1} remains valid for the descent case, although there is a subtle difference.

\begin{thm}\label{22}
	Let $\mathcal X$ be a real or complex Banach space with $\dim \mathcal X\geq 3$. Assume that $\phi : \mathcal B(\mathcal X) \rightarrow \mathcal B(\mathcal X)$ is  a    map with range containing $\mathcal F_3(\mathcal X)$.
	Then $\phi$ satisfies
	$$ \operatorname{desc}(AB + BA) = \operatorname{desc}(\phi(A) \phi(B) + \phi(B) \phi(A))\quad {\rm for\ all}\ \ A, B \in {\mathcal B}(\mathcal{X}),$$
	if and only if one of the following statements holds:

		{\rm (1)} if $ \mathcal X$ is infinite-dimensional, then there exist an invertible bounded linear or conjugate-linear operator $T : \mathcal X \to \mathcal X$ and a map $f : \mathcal B(\mathcal X) \rightarrow\mathbb{F}^* $ such that
		$$\phi(A) = f(A) TAT^{-1}\ \ {\rm for\ all} \ \ A \in \mathcal B(\mathcal X),$$
		or there exists an invertible bounded linear or conjugate-linear operator $T: \mathcal X^* \to \mathcal X$ and  a map $f: \mathcal B(\mathcal X) \to \mathbb F^*$ such that
		$$\phi(A) = f(A) TA^*T^{-1}\ \ {\rm for\ all}\ \ A \in \mathcal B(\mathcal X),$$
		and the second case occurs only when $\mathcal X$ is reflexive and   $\operatorname{desc}(C) = \operatorname{desc}(C^*)$ for all $C\in\mathcal B(\mathcal X);$
	
	{\rm (2)}  if $ \mathcal X$ is finite-dimensional, regarding $\mathcal B(\mathcal X)$ as $\mathcal M_n(\mathbb F)$,  then there exist  an invertible matrix $T\in \mathcal M_n(\mathbb F)$, an automorphism $ \tau:\mathbb{F}\rightarrow\mathbb F $ and a map $f: \mathcal M_n(\mathbb F) \to \mathbb{F}^* $ such that either
	$$\phi(A)= f(A)T[\tau(a_{ij})]T^{-1}\ \ {\rm for\ all}\ \  A=[a_{ij}] \in \mathcal M_n(\mathbb F), $$
	or
	$$\phi(A)= f(A)T[\tau(a_{ij})]^{\operatorname{tr}} T^{-1}\ \ {\rm for\ all}\ \   A=[a_{ij}] \in \mathcal M_n(\mathbb F),$$
	where $A^{\operatorname{tr}}$ denotes the transpose of $A$.
\end{thm}

\begin{proof}   The proof is similar to that of Theorem 5.1 by Remark \ref{rem4.9}, except that we cannot exclude the ring anti-isomorphism case in the infinite dimensional case  because we have no an analogous of Lemma 5.2 for descent. Clearly, this case cannot occur if $\mathcal X$ is not reflexive or if $\operatorname{desc}(A^*)\neq\operatorname{desc}(A)$ for some $A\in\mathcal B(\mathcal X)$.
\end{proof}

\begin{rem}   For the case $\dim \mathcal X=\infty$ and $\mathcal X$ is reflexive, though we are unable to determine in general whether a descent preserver of Jordan products can take the form
		$\phi(\cdot)=f(\cdot)T(\cdot)^*T^{-1}$, this form does not occur  for many Banach spaces such as those on which there exists a left invertible bounded linear operator $C$ that is not invertible. In fact, if $C$ is left invertible but not invertible, then  $\operatorname{desc}(C)=\infty$ while $\operatorname{desc}(C^*)=0$. Particularly, we have 

\if false	

 For
		instance, on $\ell^2$, let $C$ be the backward shift operator. Then
		$\operatorname{desc}(C)=0$ and $\operatorname{desc}(C^*)=\infty$.
		If $\phi(\cdot)=f(A)T(\cdot)^*T^{-1}$, then, by taking $B=I$,
		$\operatorname{desc}(CI+IC)=\operatorname{desc}(C)$,
		while
		$\operatorname{desc}\bigl(\phi(C)\phi(I)+\phi(I)\phi(C)\bigr)
		=\operatorname{desc}(C^*)$,
		a contradiction. Hence such a form cannot occur on $\ell^2$.}\fi
\end{rem}

\begin{cor}\label{5.5}
	Let $\mathcal H$ be an infinite-dimensional Hilbert space over real or complex field $\mathbb F$ and   $\phi : \mathcal B(\mathcal H) \rightarrow \mathcal B(\mathcal H)$ be a    map with range containing $\mathcal F_3(\mathcal X)$.
	Then $\phi$ satisfies
	$$ \operatorname{desc}(AB + BA) = \operatorname{desc}(\phi(A) \phi(B) + \phi(B) \phi(A))\quad {\rm for\ all}\ \ A, B \in {\mathcal B}(\mathcal{H})$$
	if and only if  there exist an invertible bounded linear or conjugate-linear operator $T : \mathcal H \to \mathcal H$ and a functional  $f : \mathcal B(\mathcal H) \rightarrow\mathbb{F}^* $ such that
		$$\phi(A) = f(A) TAT^{-1}\ \ {\rm for\ all} \ \ A \in \mathcal B(\mathcal H).$$
\end{cor}

{\bf Acknowledgments.}
 This work is partially supported by  the National Natural Science Foundation of China (Grants Nos. 12571138,  12171290,  12071336)  and Shanxi Scholarship Council of China (No. 2025-001).

{\bf Conflict of interest.} There are no conflicts of interest.

{\bf Author Contributions.}  All authors contributed to the study equally.


\end{document}